\documentclass[11pt]{article}
\usepackage{amsmath, graphicx, amsfonts,amssymb, calrsfs}
\usepackage{amsfonts,mathrsfs, color, amsthm}
\usepackage{bm}
\usepackage{dsfont}
\usepackage{subcaption}
\usepackage{float}

\usepackage{enumitem}
\usepackage{hyperref}
\hypersetup{
colorlinks   = true, 
urlcolor     = blue, 
linkcolor    = blue, 
citecolor   = red 
}

\def\N{\mathbb{N}}

\def\R{\mathbb{R}}
\def\A{\mathbf{A}}
\def\K{\mathbf{K}}
\def\E{\mathbf{E}}
\def\L{\mathbf{L}}
\numberwithin{equation}{section}

\newtheorem{theorem}{Theorem}[section]
\newtheorem{lemma}[theorem]{Lemma}
\newtheorem{remark}[theorem]{Remark}
\newtheorem{proposition}[theorem]{Proposition}
\newtheorem{corollary}[theorem]{Corollary}
\newtheorem{definition}[theorem]{Definition}

\makeatletter
\newcommand*\bigcdot{\mathpalette\bigcdot@{0.7}}
\newcommand*\bigcdot@[2]{\mathbin{\vcenter{\hbox{\scalebox{#2}{$\m@th#1\bullet$}}}}}
\makeatother

\def\bt{\begin{theorem}}
\def\et{\end{theorem}}
\def\be{\begin{equation}}
\def\ee{\end{equation}}
\def\bl{\begin{lemma}}
\def\el{\end{lemma}}
\def\br{\begin{remark}}
\def\er{\end{remark}}
\def\bc{\begin{corollary}}
\def\ec{\end{corollary}}
\def\bd{\begin{definition}}
\def\ed{\end{definition}}
\def\bp{\begin{proposition}}
\def\ep{\end{proposition}}

\def\u{\pmb{u}}
\def\v{\pmb{v}}
\def\det{\mathrm{det}}
\def\dist{\text{dist}}
\def\M{M_{n,m}(\mathbb{R})}
\def\Mn{M_{n,1}(\mathbb{R})}
\def\Mm{M_{1,m}(\mathbb{R})}

\def\x{\pmb{x}}
\def\y{\pmb{y}}
\def\tr{\mathrm{tr}}

\begin{document}
\title{On the $m$th order $p$-affine capacity
}
\author{Xia Zhou and Deping Ye}
\date{}
\maketitle
\begin{abstract} 
Let $M_{n, m}(\mathbb{R})$ denote the space of $n\times m$ real matrices, and   $\mathcal{K}_o^{n,m}$ be the set of convex bodies in $M_{n, m}(\mathbb{R})$ containing the origin.  We develop a theory for the $m$th order $p$-affine capacity $C_{p,Q}(\bigcdot)$ for $p\in[1,n)$ and $Q\in\mathcal{K}_{o}^{1,m}$. Several equivalent definitions for the $m$th order $p$-affine capacity will be provided, and some of its fundamental properties will be proved, including for example,  translation invariance and affine invariance.  We also establish several inequalities related to the $m$th order $p$-affine capacity, including those comparing to  the $p$-variational capacity,  the volume,   the $m$th order $p$-integral affine surface area, as well as the $L_p$ surface area.     


\end{abstract}

\maketitle

\section{Introduction}
The isoperimetric and isocapacity inequalities are fundamental in mathematics, physics, and related fields. In their simplest form (see, e.g., \cite{Gardner-2002-LP-isop-Ine,Ludwig Xiao and Zhang 2011,Lutwak1,Mazya 1985,Mazya 2011,schneider}), these inequalities can be stated as follows:
for $p\in [1, n)$ and for any convex body $K\subset \Mn$ (i.e., a compact convex set with nonempty interior) containing the origin $o$ in its interior, one has
\begin{align}
\label{xiaos-result-2--1}
\left( \frac{V_n(K)}{V_n(B_2^n)} \right)^{\frac{1}{n}} \leq \left( \frac{C_p(K)}{C_p(B_2^n)} \right)^{\frac{1}{n - p}} \leq \left( \frac{S_p(K)}{S_p(B_2^n)} \right)^{\frac{1}{n - p}}.
\end{align}
Hereafter, $\Mn$ denotes the $n$-dimensional Euclidean space, $B_2^n$ refers to the unit Euclidean ball in $\Mn$,  and $V_n(K)$ and $S_p(K)$ are the volume and $L_p$ surface area (see \eqref{lp-surface-area-of-star body}) of $K$, respectively. By $C_p(K)$, we mean the $p$-variational capacity of $K$ (see e.g., Evans and Gariepy \cite{Evans-Gariepy-1992}, and Maz'ya \cite{Mazya 1985,Mazya 2011}), which may be formulated by   
\begin{align*}
C_p(K)=\inf\bigg\{\int_{\Mn}|\nabla f(x)|^p d x:  f \in C_c^{\infty}\ \mathrm{and}\  f \geq 1\ \ \mathrm{on}\  K\bigg\},
\end{align*} where $C_c^{\infty}$ denotes the set of compactly supported, infinitely continuously differentiable functions on $\Mn$, $\nabla f$ is the gradient of $f$, $|y|$ is the Euclidean norm of $y\in \Mn$, and $dx$ denotes the Lebesgue measure on $\Mn$. 
Inequality \eqref{xiaos-result-2--1} is  closely related to the fundamental $L_p$ Sobolev inequality  (see e.g., Federer and Fleming \cite{Federer-p-1} and Maz'ya \cite{Maz'ya-p-large-1} for $p=1$, Aubin \cite{sharp-lp-sobolev-Aubin-p-ge-1} and Talenti \cite{Talenti-p-large-1} for $p\in(1,n)$): for $p\in [1, n)$ and for $f\in C_c^{\infty}$, there exists a constant $C>0$, depending only on $p$ and $n$, such that 
\begin{align}
\label{Lp-sov-11} 
    C\|f\|_{\frac{np}{n - p}}\leq  \|\nabla f\|_{p}, 
\end{align} where $\|f\|_p$ and $\|\nabla f\|_p$ are the $L_p$ norms of $f$ and $\nabla f$, respectively, given by $$\|f\|_p=\Big(\int_{\Mn} |f(x)|^pdx\Big)^{\frac{1}{p}}\ \ \mathrm{and} \ \ \|\nabla f\|_p=\Big(\int_{\Mn} |\nabla f(x)|^pdx\Big)^{\frac{1}{p}}.$$ In particular, the first inequality in \eqref{xiaos-result-2--1} follows from \eqref{Lp-sov-11} by taking the infimum over all $f\in C_c^{\infty}$ such that $f\geq 1$ on $K$. 
Inequality \eqref{Lp-sov-11} also extends to more general functions in the Sobolev space $W^{1, p}$, as defined in Section \ref{notation}. Note that the $p$-variational capacity plays a central role in various  problems in convex geometry and analysis, including Minkowski-type problems for the $p$-variational capacity (see e.g., \cite{Colesanti 2015,Hong-Ye-Zhang-2018,Jerison 1996,Jerison-1996-2,Ji-dual-Minkowski-p-for-p-capacity,Liu-Sheng-flow-Minkowski-problem-q-capacity,Xiong-Xiong-Lu-2019,zou-xiong-2020}), and the Brunn-Minkowski inequalities for the $p$-variational capacity (see e.g., \cite{Borell 1983,Caffarelli 1996,for p capacity 1,Hu 2024,Zou and Xiong 2018}).

While inequality \eqref{xiaos-result-2--1} reveals the relationships between volume, $p$-variational capacity, and $L_p$ surface area, the latter two are not invariant under volume-preserving linear transformations. Consequently, these inequalities may be arguably less powerful compared to their affine invariant counterparts, namely, the affine isoperimetric, isocapacity and Sobolev inequalities (see e.g., \cite{Ludwig Xiao and Zhang 2011,lp proj ineq 1,LYZ-2000-a-new ellipsoid, sharpaffine lp sobol ine,Petty isoperimetric problem,Xiao 2015,Xiao 2016,Zhang 1999}). 
Let $\M$ denote the space of $n\times m$ real matrices. Let $K\subset \Mn$ be a convex body containing $o$ in its interior, and let $f\in C_c^{\infty}$. For convenience of later context, we let 
\begin{align}\label{Phi-p-H}
 \Phi_{p,H}(K) &=\Big(\int_{S^{nm-1}}\Big(\int_{S^{n-1}}H(v^{\mathrm{T}}{}_{\bigcdot} \u)dS_{K,p}(v)\Big)^{-\frac{nm}{p}} d \u\Big)^{-\frac{p}{nm}} \ \ \ \  \mathrm{for}\ \ p\in[1,\infty),\\ \label{E-p-H}
  \mathcal{E}_{p, H}(f) & = \Big(\int_{S^{n m-1}}\Big(\int_{\Mn} H\big((\nabla f(z))^{\mathrm{T}}{}_{\bigcdot} \u\big) d z\Big)^{-\frac{n m}{p}} d \u\Big)^{-\frac{1}{n m}} \ \ \ \   \mathrm{for}\ \ p\in[1,\infty),
\\ \label{C-p-H}
 C_{p,H}(K) & =\inf\Big\{\mathcal{E}_{p,H}^p(f):\ \ f \in C_c^{\infty}\ \mathrm{and}\ \!f \geq 1\ \mathrm{on}\  K\Big\} \ \ \ \  \mathrm{for}\ \ p\in[1,n),
\end{align}
where $H:\Mm\rightarrow [0,\infty)$ is a continuous function, $\Mm$ denotes the $1\times m$ real matrix space, $d S_{K,p}$ is the $L_p$ surface area measure of $K$ (see Section \ref{notation} for details), $d\u$ is the  spherical measure on the unit sphere $S^{nm-1}$ and $v^{\mathrm{T}}{}_{\bigcdot} \u$ denotes the matrix product of the transpose of $v$ and $\u$. If $m=1$ and $H_1(t)=\big(\frac{|t|}{2}\big)^p$ for $t\in\R$, then $v^{\mathrm{T}}{}_{\bigcdot} \u=v\bigcdot \u$  is the inner product of $v,\u\in \Mn,$ and  we get the $p$-integral affine surface area $\Phi_{p}(K)$ associated to projection bodies (Petty \cite{Petty-projection-bodies,Petty isoperimetric problem} for $p=1$, and Lutwak, Yang and Zhang \cite{lp proj ineq 1} for $p\in[1,\infty)$), the $L_p$ affine energy $\mathcal{E}_p(f)$ (Zhang \cite{Zhang 1999} for $p=1$, and Lutwak, Yang and Zhang \cite{sharpaffine lp sobol ine} for $p\in[1,\infty)$), and the $p$-affine capacity $C_{p, 0}(K)$ by Xiao \cite{Xiao 2015,Xiao 2016} for $p\in [1,n)$, respectively:
\begin{align}\label{phi-e-c-0}
\Phi_{p}(K)=\Phi_{p,H_1}(K), \ \ \mathcal{E}_{p}(f)=\mathcal{E}_{p, H_1}(f) \ \ \mathrm{and}\ \ C_{p,0}(K)=C_{p,H_1}(K).
\end{align}
The affine isoperimetric and affine isocapacity inequalities can be stated together as follows: for $p\in [1, n)$ and for any convex body $K\subset \Mn$ containing $o$ in its interior, one has 
\begin{align}
  \label{xiaos-result-1--1}
&\bigg(\frac{V_n(K)}{V_n(B_2^n)}\bigg)^{\frac{1}{n}}
\leq\bigg(\frac{C_{p, 0}(K)}{C_{p, 0}(B_2^n)}\bigg)^{\frac{1}{n-p}} \leq
\bigg(\frac{\Phi_p(K)}{\Phi_p(B_2^n)}\bigg)^{\frac{1}{n-p}}. 
\end{align} 
Associated to the $L_p$ affine energy is the $L_p$ affine Sobolev inequality (see e.g., \cite{sharpaffine lp sobol ine,Zhang 1999}): for $p\in [1, n)$ and $f\in C_c^{\infty}$, there exists a constant $C_{n,p}$, depending only on $n$ and $p$, such that 
\begin{align}
    \label{Lp-aff-sov-11}  
C_{n, p}  \|f\|_{\frac{np}{n - p}} \leq \mathcal{E}_{p}(f), 
\end{align}  which is the affine invariant counterpart of  \eqref{Lp-sov-11} and is closely related to inequality  \eqref{xiaos-result-1--1}. 
For example, the first inequality in \eqref{xiaos-result-1--1} can be obtained from \eqref{Lp-aff-sov-11}  by taking the infimum over $f\in C_c^{\infty}$  such that $f\geq 1$ on $K$. 
We would like to mention that the above inequalities  \eqref{xiaos-result-1--1} and \eqref{Lp-aff-sov-11} may work for other range of $p$ and more general $f$ and $K$.
Note that the $L_p$ affine energy $\mathcal{E}_p(\bigcdot)$ has been the key ingredient for various affine functional inequalities \cite{Cianchi-LYZ-2009,lp affine appl 2,Haddad-Ludwig-affine-frac-sobolev-isop-ine,sharpaffine lp sobol ine,Wang-tuo-2012,Wang-2013,Zhang 1999} among others. See e.g., Xiao \cite{Xiao-2017} and Xiao and Zhang \cite{Xiao-Zhang-2016} for more contributions on the $p$-affine capacity. Moreover, the aforementioned affine inequalities are stronger than their Euclidean relatives;  see e.g., Ludwig, Xiao and Zhang \cite{Ludwig Xiao and Zhang 2011}, and   Lutwak, Yang and Zhang \cite{sharpaffine lp sobol ine}:
\begin{align}
\label{xiaos-result-2}
 \frac{C_{p, 0}(K)}{C_{p, 0}(B_2^n)} \leq \frac{C_p(K)}{C_p(B_2^n)}\ \ \mathrm{for} \ \ p\in [1, n)\  \ \ \mathrm{and} \ \  \ \frac{\Phi_p(K)}{\Phi_p(B_2^n)} \leq \frac{S_p(K)}{S_p(B_2^n)} \ \ \mathrm{for} \ \ p\in [1, \infty).
\end{align}
 
Both $C_{p, 0}(\bigcdot)$ and   $\Phi_p(\bigcdot)$ have been extended to their asymmetric counterparts involving the function $\varphi_{\tau}$ with  $\tau\in [-1, 1]$ given by:    \begin{align}\label{varphi tau}
\varphi_\tau(t)^p=\Big(\frac{1+\tau}{2}\Big) t_{+}^p+\Big(\frac{1-\tau}{2}\Big) t_{-}^p,
\end{align} where $t_{+}=\max \{0, t\}$ and $t_{-}=\max \{0,-t\}$ for $t\in\mathbb{R}$. If $m=1$ and $H_2(t)=\varphi_{\tau}(t)^p$ for $t\in\R$, then we get the general $p$-integral affine surface area $\Phi_{p,\tau}(K)$ and asymmetric $L_p$ affine energy $\mathcal{E}_{p,\tau}(f)$ by Haberl and Schuster \cite{lp-affine-isoperimetric-ine, asym affine lp sobo ine}, and general $p$-affine capacity $C_{p, \tau}(K)$ by Hong and Ye in \cite{Hong and Ye 2018}:
\begin{align*}
\Phi_{p,\tau}(K)=\Phi_{p,H_2}(K), \ \ \mathcal{E}_{p,\tau}(f)=\mathcal{E}_{p, H_2}(f) \ \ \mathrm{and}\ \ C_{p,\tau}(K)=C_{p,H_2}(K).
\end{align*}
These are equivalent to \eqref{phi-e-c-0} when $\tau=0$ up to a multiplicative constant.
The asymmetric versions of inequalities \eqref{xiaos-result-2--1}, \eqref{xiaos-result-1--1}   and \eqref{xiaos-result-2} can be summarized as follows (see e.g., \cite{lp-affine-isoperimetric-ine, Hong and Ye 2018}): for $p\in [1, n)$, $\tau\in [-1, 1]$ and $K$ a convex body containing $o$ in its interior, one has  
\begin{align}
\label{hh and ye cp}
&\bigg(\frac{V_n(K)}{V_n(B_2^n)}\bigg)^{\frac{1}{n}}
\leq
\bigg(\frac{C_{p, \tau}(K)}{C_{p, \tau}(B_2^n)}\bigg)^{\frac{1}{n-p}}
\leq
\bigg(\frac{C_p(K)}{C_p(B_2^n)}\bigg)^{\frac{1}{n-p}}
\leq
\bigg(\frac{S_p(K)}{S_p(B_2^n)}\bigg)^{\frac{1}{n-p}},\\
\label{hh and ye phi p} 
&\bigg(\frac{V_n(K)}{V_n(B_2^n)}\bigg)^{\frac{1}{n}} \leq\bigg(\frac{C_{p, \tau}(K)}{C_{p, \tau}(B_2^n)}\bigg)^{\frac{1}{n-p}} \leq\bigg(\frac{\Phi_{p, \tau}(K)}{\Phi_{p, \tau}(B_2^n)}\bigg)^{\frac{1}{n-p}} \leq\bigg(\frac{S_p(K)}{S_p(B_2^n)}\bigg)^{\frac{1}{n-p}}. 
\end{align} 
Again, some of these inequalities may work for other range of $p$ and more general $f$ and $K$. Moreover, these inequalities are closely related to the asymmetric $L_p$ affine Sobolev inequalities  by Haberl and Schuster \cite{asym affine lp sobo ine}. 
See e.g., \cite{Haberl-Schuster-Xiao-2012,Haddad-Ludwig-affine-frac-lp-sobolev-ine,Nguyen-2016,Wang-2015} for more applications of $\mathcal{E}_{p, \tau}(\bigcdot).$

Our main goal in this paper is to develop a theory of the $m$th order $p$-affine capacity and establish related affine isoperimetric and isocapacity inequalities that extend \eqref{hh and ye cp} and \eqref{hh and ye phi p}.  These are motivated by the recent progress on the $m$th order affine isoperimetric and Sobolev inequalities by Haddad, Langharst, Putterman, Roysdon and Ye \cite{Haddad-Ye-2023,Haddad-ye-lp-2025}. This direction dates back to Schneider \cite{Schneider 0}, where he introduced the $m$th order difference body and established the corresponding Rogers-Shephard inequality. The main idea behind Schneider's work is the introduction of $m$th order covariogram function of a convex body, whose support is naturally the $m$th order difference body. Such a covariogram function has been used in \cite{Haddad-Ye-2023,Haddad-ye-lp-2025} to define the $(L_p,Q)$-projection body $\Pi_{p, Q}K$ (see \eqref{support function of ho lp projection body} for the formula of its support function), where $p\geq 1$,  $Q\subset \Mm$ is a convex body containing $o$ and $K\subset \Mn$ is a convex body with $o$ in its interior. They also established the higher order $L_p$ affine isoperimetric inequalities, which take the following form: 
\begin{align}\label{conclusion result 2---1}
\left(\frac{V_n(K)}{V_n(B_2^n)}\right)^{\frac{1}{n}}   \leq\left(\frac{\Phi_{p, Q}(K)}{\Phi_{p, Q}(B_2^n)}\right)^{\frac{1}{n-p}}, 
\end{align} where $\Phi_{p,Q}(K)$ is the $m$th order $p$-integral affine surface area of $K$ given by \eqref{Phi-p-H} with $H_3(x)=h_Q(x)^p$ for $x\in \Mm$. Here $h_Q(x)$ is the support function of $Q$ defined in \eqref{def-support-f}.   Another important object defined in \cite{Haddad-Ye-2023,Haddad-ye-lp-2025} is the $m$th order $L_p$ affine energy $\mathcal{E}_{p, Q}(f)$ for $f\in C_c^{\infty}$, defined by \eqref{E-p-H} again with the same choice $H_3$. The following $m$th order $L_p$ affine Sobolev inequality has been established in \cite[Corollary 1.10]{Haddad-Ye-2023} for $p=1$ and \cite[Theorem 1.1]{Haddad-ye-lp-2025} for $p\in[1,n)$. 
Let $Q\subset\Mm$ be a convex body containing $o$ and let $f \in W_0^{1, p}$ (see Section \ref{notation} for the definition) be a non-constant function. Then there exists a constant $C_{n,p,Q}>0$, depending only on $n,p$ and $Q$, such that  \begin{align} 
 C_{n,p,Q}\|f\|_{\frac{np}{n- p}}\leq \mathcal{E}_{p, Q}(f).\label{high-ord-Sov-ineq-1}
\end{align}
When $m=1$ and 
 $$Q=\Big[-\Big(\frac{1-\tau}{2}\Big)^{\frac{1}{p}},\ \Big(\frac{1+\tau}{2}\Big)^{\frac{1}{p}}\Big]\subset \R \ \ \mathrm{for}\ \ \tau\in [-1,1],$$ 
then  $h_Q^p(\bigcdot)=\varphi_{\tau}(\bigcdot)$, and hence $C_{p, Q}(\bigcdot)$ reduces to the  $p$-affine capacity \cite{Hong and Ye 2018,Xiao 2015, Xiao 2016}. Consequently, inequalities \eqref{conclusion result 2---1} and \eqref{high-ord-Sov-ineq-1} extend the corresponding inequalities, such as,  \eqref{hh and ye phi p} and the asymmetric $L_p$ affine Sobolev inequalities. The $m$th order theory has attracted a lot of attention recently, see e.g.,  \cite{J. Haddad,Haddad-Putterman-2025,Dylan-2025-Some-Comments,Dylan-2025-moment-entropy,mth order weighted projection body,mth order affine polya szego principle,higher order reverse isoperimetric inequality,higher order lp mean zonoids,Sola-on-general-versions-of-projection-ine,the-mth-order-Orlicz-projection-bodies}. 

In view of the relations between affine capacities and the affine energies, the $m$th order $p$-affine capacity can be introduced and applied to get stronger or more general inequalities than  \eqref{hh and ye cp}, \eqref{hh and ye phi p} and \eqref{conclusion result 2---1}. Several equivalent definitions for the $m$th order $p$-affine capacity are provided in Section \ref{section-3}, and here we present one of them below. Denote by $\mathcal{K}_{o}^{1,m}$ the set of convex bodies in $\Mm$ containing $o$. Let $p\in [1,n)$, $Q\in\mathcal{K}_o^{1,m}$, and $K\subset \Mn$ be compact. The $m$th order $p$-affine capacity of $K$, $C_{p, Q}(K)$, is given by \eqref{C-p-H} with $H(x)=h_Q(x)^p$ for $x\in \Mm$, i.e.,
\begin{align*}
C_{p, Q}(K)=\inf\Big\{ \mathcal{E}_{p, Q}^p(f)\ : \ f \in C_c^{\infty}\ \mathrm{and}\  f \geq 1\  \mathrm{on}\  K\bigg\}.
\end{align*}
We also prove some basic properties for the $m$th order $p$-affine capacity, including symmetry and concavity in $Q$, as well as monotonicity, homogeneity, sub-additivity, translation invariance and boundary behavior in $K$. Furthermore, we show that $C_{p,Q}(\bigcdot)$ satisfies both the upper semi-continuity and the continuity from above.

One of the main results in Section \ref{section-4}  is to find the precise value of the $m$th order $p$-affine capacity for the unit Euclidean ball. This is stated in Theorem  \ref{theorem with volume-1} and we provide its precise value below for readers' convenience: for $Q\in\mathcal{K}_o^{1,m}$,  one has 
\begin{align*} 
C_{p,Q}(B_2^n)= \begin{cases}
0, &\ \ \mathrm{for}\ \ p\ge n,\\
n V_n(B_2^n) d_{n, p}(Q)\big(\frac{n-p}{p-1}\big)^{p-1}, &\ \ \mathrm{for}\ \ p\in (1,n), \\
n V_n(B_2^n)d_{n, p}(Q), &\ \ \mathrm{for}\ \ p=1.
\end{cases}   
\end{align*} The inequalities to compare the $m$th order $p$-affine capacity with the volume and the $p$-variational capacity are proved in Theorem \ref{theorem with volume} and, respectively, in Theorem \ref{cpqk and cpk}. These inequalities can be summarized as follows:  for $p\in [1,n)$,  $Q\in\mathcal{K}_o^{1,m}$,  and $K\subset \Mn$   a convex body containing $o$ in its interior, one has, 
\begin{align*}
\left(\frac{V_n(K)}{V_n(B_2^n)}\right)^{\frac{1}{n}}
\leq
\left(\frac{C_{p, Q}(K)}{C_{p, Q}(B_2^n)}\right)^{\frac{1}{n-p}}
\leq
\left(\frac{C_p(K)}{C_p(B_2^n)}\right)^{\frac{1}{n-p}}
\leq
\left(\frac{S_p(K)}{S_p(B_2^n)}\right)^{\frac{1}{n-p}}. 
\end{align*} These inequalities are stated for $K$ being a convex body, but they also hold for more general compact sets. See more details in  Section  \ref{section-4}.   

Section \ref{Section-5} is dedicated to proving inequalities comparing volume, $m$th order $p$-affine capacity,  $m$th order $p$-integral surface area, and $L_p$ surface area. These results are established for Lipschitz star bodies $K$ (see Section \ref{notation} for the definition) and are given in Theorems \ref{theorem c and phi when p ge 1} and \ref{c1Qk and Phi 1 QK}. For reader's convenience, we only state them for $K$ being a convex body containing $o$ in its interior: for $p\in [1, n)$ and $Q\in\mathcal{K}_o^{1,m}$, one has   \begin{align*} 
\left(\frac{V_n(K)}{V_n(B_2^n)}\right)^{\frac{1}{n}} \leq\left(\frac{C_{p, Q}(K)}{C_{p, Q}(B_2^n)}\right)^{\frac{1}{n-p}} \leq\left(\frac{\Phi_{p, Q}(K)}{\Phi_{p, Q}(B_2^n)}\right)^{\frac{1}{n-p}} \leq\left(\frac{S_p(K)}{S_p(B_2^n)}\right)^{\frac{1}{n-p}}.    
\end{align*}

\section{Preliminaries and Notations}\label{notation}
Our setting in this paper is the real vector space of $n\times m$ matrices, denoted by $\M$, with  $n,m\in \mathbb{N}$, where $\mathbb{N}$ is the set of positive integers.  When no matrix product is involved, $\M$ is often  identified with the Euclidean space $\mathbb{R}^{nm}$. For instance, $\Mn$ and $\Mm$ are identified as  $\mathbb{R}^n$ and $\mathbb{R}^m$, respectively. Elements of $\Mn$ and $\Mm$ are usually written by lowercase letters, such as $x$ and $y$, while elements of $\M$, for $m\neq 1$ and $n\neq 1$, are written by boldface letters, such as $\x$ and  $\y$. Boldface uppercase letters, such as $\A$, $\mathbf{K}$, $\mathbf{L}$, are used to denote subsets of $\M$. 

The space $\M$ is equipped with the standard inner product:  $\x\bigcdot\y=\tr(\x^{\mathrm{T}}{}_{\bigcdot}\y),$
where $\tr(\bigcdot)$ is the trace operator and  $\x^{\mathrm{T}}{}_{\bigcdot}\y\in M_{m,m}(\mathbb{R})$ is the matrix product of $\x^{\mathrm{T}}$, the transpose of  $\x \in M_{n,m}(\mathbb{R})$, and  $ \y\in M_{n,m}(\mathbb{R})$. When $\x=\y$, one gets the norm of $\x\in\M$ given by $|\x|=\sqrt{\x\bigcdot\x}$. The distance from $\x \in \M$ to $\A \subset \M$ is given by
\begin{align*}
\dist(\x, \A)=\inf \big\{|\x-\y|:\   \y \in \A\big\}.
\end{align*}
Clearly, $\dist(\x, \A)=0$  if $\x \in \overline{\A}$, the closure of $\A$. For $\A\subset \M$, let  $\mathbf{1}_\A(\bigcdot)$ denote the indicator function of $\A$, defined by $\mathbf{1}_{\A}(\x)=1$ if $\x\in \A$ and $\mathbf{1}_{\A}(\x)=0$ if $\x\notin \A$.

 A set $\K\subset \M$ is a convex body if it is compact, has a nonempty interior, and satisfies the property that the line segment joining any two points in $\K$ lies entirely within $\K$. Associated to a convex body $\K\subset \M$ is the support function $h_{\K}: \M\rightarrow \R$ defined by \begin{align}\label{def-support-f}
     h_{\K}(\x)=\max\big\{\x\bigcdot \y : \y\in \K\big\} \ \ \mathrm{for} \ \ \x\in \M.
 \end{align}  Clearly,  for $A\in M_{n,l}(\mathbb{R})$, $\x\in M_{l,m}(\mathbb{R})$, $B\in M_{l,m}(\mathbb{R})$ and $\y\in M_{n,l}(\mathbb{R})$ with $l\in\N$, one has,
\begin{align}
\label{h KB} 
h_{\K}(A{}_{\bigcdot}\x)=h_{A^{\mathrm{T}}{}_{\bigcdot}{\K}}(\x)\ \ \mathrm{and}\ \
h_{\K}(\y{}_{\bigcdot}\, B)&=h_{\K{}_{\bigcdot}\, B^{\mathrm{T}}}(\y).
\end{align} In particular,   for any $c>0$ and $\x\in \M$, it holds that 
\begin{align}
\label{hom-1-1}
h_{\K}(c\x)= h_{c\K}(\x)=ch_{\K}(\x).
\end{align} Note that $h_{\K}$ is Lipschitz continuous on every compact subset $\E\subset\M$, i.e., for each $\E$, there exists a constant $\mathrm{Lip}_{\E}>0$ such that $$|h_{\K}(\x)-h_{\K}(\y)|\leq \mathrm{Lip}_{\E}\  |\x-\y| \ \ \mathrm{for\ all} \ \ \x,\y\in \E.$$ 

Let $\partial \K$ be the boundary of the convex body $\K\subset \M$.  Write  $\nu_{\K}(\bigcdot)$ for the Gauss map of $\K$, and hence $\nu_{\K}(\y)$ represents the set of the outer unit normal vectors at $\y\in\partial \K$. The preimage of $\nu_{\K}(\bigcdot)$, denoted by $\nu_{\K}^{-1}(\bigcdot)$, represents the reverse Gauss map of $\K$. Let $S^{nm-1}$ be the unit Euclidean sphere in $\M$.   Associated to $\K$ is a Borel measure $S_{\K}(\bigcdot)$,   the surface area measure of $\K$, defined by \begin{align*}
\int_{S^{nm-1}}f(\u)dS_{\K}(\u)=\int_{\partial \K} f\big(\nu_{\K}(\pmb{y})\big)d\mathcal{H}^{nm-1}(\pmb{y}),
\end{align*} for any continuous function $f:S^{nm-1}\to \R$. Hereafter,  $\mathcal{H}^{nm-1}$ stands for the $(nm-1)$-dimensional Hausdorff measure. In fact, for any Borel subset $\omega\subset S^{nm-1},$ one has    
\begin{align*}
S_{\K}(\omega)=\mathcal{H}^{mn-1}\big(\nu_{\K}^{-1}(\omega)\big).  
\end{align*}  

Let $\mathcal{K}_{(o)}^{n,m}$ be the set consisting of all convex bodies in $\M$ containing the origin $o$ in their interiors. For each $\K\in\mathcal{K}_{(o)}^{n,m}$, there exists a convex body $\K^*\in\mathcal{K}_{(o)}^{n,m}$, called the polar body of $\K$, defined by
\begin{align}\label{polar-body}
\K^*=\big\{\x\in \M: \x\bigcdot \y\leq1 \ \ \mathrm{for\ any}\ \ \y\in \K\big\}.
\end{align}
It is well known that, for each $\K\in \mathcal{K}_{(o)}^{n,m}$,  \begin{align}\label{VnmKstar}
V_{nm}(\K^*)=\frac{1}{nm}\int_{S^{nm-1}} \Big(\frac{1}{h_{\K}(\u)}\Big)^{nm}d\u, 
\end{align}  where the integral is taken with respect to $d\u=d\mathcal{H}^{nm-1}|_{S^{nm-1}}$, and $V_{nm}(\E)$ denotes the volume of $\E\subset  \M$ whenever it exists. 

A set $\L\subset M_{n,m}(\mathbb{R})$ is called a star-shaped set about $o\in M_{n,m}(\mathbb{R})$,  if $\lambda \y \in \L $
holds for all $\lambda\in [0, 1]$ and $\y\in \L.$ For $\L$ a star-shaped set about $o$, its radial function $\rho_{\L}$ can be formulated by 
$$
\rho_{\L}(\x)=\sup \{\lambda\ge0:\lambda \x\in \L\} \ \  \mathrm{for} \ \ \x\in \M\setminus\{o\}.
$$ Such a star-shaped set is called a star body about $o$ if  $\rho_{\L}$, restricted to $S^{nm-1}$, is strictly positive and continuous. Again, it is well known that  
\begin{align}\label{volume-K}
V_{nm}(\L)=\int_{\L}d\x=\frac{1}{nm}\int_{S^{nm-1}} \rho_{\L}(\u)^{nm}d\u =\frac{1}{nm} \int_{\partial \L} \x \bigcdot \nu_{\L}(\x) d \mathcal{H}^{nm-1}(\x),
\end{align} where $d\x$ denotes the Lebesgue measure on $\M$, $\nu_{\L}(\x)$ is a unit outer normal of $\partial \L$ at $\x$, and the last equality follows directly from the first one by using the divergence formula. 

We call $\L$ a Lipschitz star body if it is a star body about $o$  and its boundary is Lipschitz. Let $\mathcal{L}_o^{n,m}$ stand for the set of Lipschitz star bodies about $o$ in $\M$. Clearly,  $\mathcal{K}_{(o)}^{n,m}\subset \mathcal{L}_o^{n,m}$.
Define the gauge function of $\L\in\mathcal{L}_o^{n,m}$ by  
\begin{align}\label{PL}
p_{\L}(\y)=\frac{1}{\rho_{\L}(\y)}\ \ \mathrm{if}\ \ \y \neq 0, \ \ \mathrm{and} \ \ \  p_{\L}(\y)=0\ \   \mathrm{if}\ \ \y=0.
\end{align}

A function $f:\Mn\rightarrow \R$  is said to have compact support if it vanishes outside a compact subset of $\Mn$. 
The standard notation $C_c^{\infty}\big(\Mn\big)$ (or simply $C_c^{\infty}$) refers to the space of smooth (i.e., infinitely differentiable) functions on $\Mn$ with compact supports. For $p\in [1,\infty)$, denote by $L_p\big(\Mn\big)$ (or simply $L_p$) the space of all Lebesgue measurable functions $f:\Mn\rightarrow \mathbb{R}$ such that the $L_p$ norm
\begin{align}\label{lp-norm}
\|f\|_{p}=\Big(\int_{\Mn}|f(x)|^p d x\Big)^{\frac{1}{p}}<\infty. \end{align} 

 Let $U\subset \Mn$ be a nonempty open set. 
The space $L_{loc}^1(U)$ (see, e.g., Evans and Gariepy \cite[Definition 4.1]{Evans-Gariepy-1992}) consists of all measurable functions that are integrable on every compact subset of $U$.
We say that the weak partial derivative $\partial_i f$ of $f \in L_{loc}^1(U)$ exists if there exists a function $g_i \in L_{loc}^1(U)$ such that
$$
\int_{U}f(x)\partial_i \varphi(x)dx=-\int_{U} g_i(x) \varphi(x) d x \ \ \mathrm{for \ all \ } \varphi\in C_c^{\infty}(U).
$$
In this case, we write $\partial_if=g_i$ and call the vector $\nabla f=\left(\partial_1 f, \cdots, \partial_n f\right)$ the weak gradient of $f$. Let $W^{1,p}\big(\Mn\big)$ (or simply $W^{1,p}$) denote the Sobolev space consisting of all  $f\in L_p$ such that $\partial_i f \in L_p$ for all $i=1,\cdots,n$ (see, e.g., Evans and Gariepy \cite[Definition 4.2]{Evans-Gariepy-1992}). A natural norm in $W^{1,p}$ is 
$$\|f\|_{1,p}=\|f\|_p+\|\nabla f\|_p,$$
where $\|\nabla f\|_p$ is the $L_p$ norm of $\nabla f$, i.e., $$\|\nabla f\|_p=\Big(\int_{\Mn} |\nabla f(x)|^pdx\Big)^{\frac{1}{p}}.$$
The completion of $C_c^{\infty}$ under the norm $\|\bigcdot\|_{1,p}$ is a subspace of $W^{1,p}$, denoted by  
$W_0^{1,p}\big(\Mn\big)$ (or simply $W_0^{1,p}$).

Let $f \in C_c^{\infty}$ and $\|f\|_{\infty}=\sup_{x\in\Mn}|f(x)|$. Then, for any $t \in\left(0,\|f\|_{\infty}\right)$, the set \begin{align}\label{level-set-f}
[f]_t=\Big\{x \in \Mn:|f(x)| \geq t\Big\},
\end{align} is compact. Note that $[f]_t$ is known as the superlevel set of $f$ and has the following property (i.e., Sard's theorem \cite[P.39]{Guillemin-1974}):  for almost all $t \in\left(0,\|f\|_{\infty}\right)$, the smooth $(n-1)$-dimensional submanifold  $$\partial[f]_t=\Big\{x \in \Mn:|f(x)|=t\Big\}$$ satisfies that $\nabla f(x) \neq o$ for all $x \in \partial[f]_t$. Moreover,
\begin{align}\label{nu-f-t}
\nu_{[f]_t}(x)=-\frac{\nabla f(x)}{|\nabla f(x)|}.   
\end{align} 
We shall need the following Federer's coarea formula \cite[Theorem 3.2.12]{Federer 1969}: if $f: \Mn \rightarrow \mathbb{R}$ is a Lipschitz function and $g: \Mn \rightarrow[0, \infty)$ is a Lebesgue integrable function, then for any open set $\Omega \subset \mathbb{R}$,
\begin{align}\label{Federer-coarea-formula}
\int_{f^{-1}(\Omega) \cap\{|\nabla f|>0\}} g(x) d x=\int_{\Omega} \int_{f^{-1}(t)} g(x)|\nabla f(x)|^{-1} d \mathscr{H}^{n-1}(x) d t.
\end{align}

\section{The $m$th order $p$-affine capacity}\label{section-3}
This section is dedicated to the definition and some fundamental properties for the $m$th order $p$-affine capacity $C_{p,Q}(K)$, where  $p\in [1,n)$, $Q\in \mathcal{K}_o^{1,m}$, and $K$ is a compact set in $\Mn$. Here $ \mathcal{K}_o^{1,m}$ is the set of convex bodies in $\Mm$ containing the origin $o$. The fundamental properties, including the symmetry and concavity in $Q$, as well as the monotonicity, homogeneity, sub-additivity, translation invariance and boundary behavior in $K$ are provided. We also prove that $C_{p,Q}(K)$ satisfies the upper semi-continuity and the continuity from above.  

We now give the definition of the $m$th order $p$-affine capacity. For $f\in W_0^{1,p}$ and $\u\in S^{nm-1}$, let
$\nabla_{\u} f=\nabla f^{\mathrm{T}}{}_{\bigcdot}\u$ and 
\begin{align}\label{ek}
\mathscr{A}(K)=\left\{f \in W_0^{1, p}: f \geq \mathbf{1}_K\right\}.
\end{align}
\begin{definition}\label{def of cpqk with ek}
Let $p\in [1,n)$, $Q\in\mathcal{K}_o^{1,m}$ and $K$ be a compact subset of $\Mn$. Define the $m$th order $p$-affine capacity of $K$ by 
\begin{align*}
C_{p,Q}(K)=\inf_{f\in\mathscr{A}(K)}\Big(\int_{S^{nm-1}}\left\|h_Q\left(\nabla_{\u} f\right)\right\|_p^{-nm} d \u\Big)^{-\frac{p}{nm}}.
\end{align*}
\end{definition}
The $m$th order $p$-affine capacity can also be extended to bounded measurable sets. In this case, if  $L \subset \Mn$ is  a bounded measurable set, then 
\begin{align*}
C_{p, Q}(L)=\inf \Big\{C_{p, Q}(O): L \subset O\ \ \mathrm{and}\ \ O\subset \Mn\ \ \mathrm{is \ bounded \ and \ open}\Big\},
\end{align*}
where for a bounded open subset $O\subset \Mn$,
\begin{align*}
C_{p, Q}(O)=\sup \Big\{C_{p, Q}(K): K \subset O \ \ \mathrm{and}\ \  K \ \mathrm{\ is\ compact}\Big\}.
\end{align*}
In this paper, we focus only on $C_{p,Q}(K)$ for compact $K\subset \Mn$, and the definition for bounded measurable sets will not be used.  Several special cases of the $m$th order $p$-affine capacity are provided. Let $\varphi_{\tau}$ for $|\tau|\leq 1$ be given in \eqref{varphi tau}.  When $m=1$ and 
$$Q=\Big[-\Big(\frac{1-\tau}{2}\Big)^{\frac{1}{p}},\ \Big(\frac{1+\tau}{2}\Big)^{\frac{1}{p}}\Big]\subset \R,$$
then  $h_Q^p(\bigcdot)=\varphi_{\tau}(\bigcdot)$. Thus $C_{p, Q}(\bigcdot)$ reduces to the general $p$-affine capacity defined by Hong and Ye \cite[Definition 3.1]{Hong and Ye 2018}. It further reduces to the $p$-affine capacity, if  $\tau=0$,  defined by Xiao \cite[P.3]{Xiao 2015} and \cite[Definition 2.1]{Xiao 2016}.

Several equivalent definitions will be provided to better understand the $m$th order $p$-affine capacity $C_{p, Q}(\bigcdot)$. 
To fulfill this goal, a key step is to establish the finiteness of $C_{p, Q}(\bigcdot)$, which can be derived from the monotonicity and homogeneity of $C_{p, Q}(\bigcdot)$. For $K\in\mathcal{K}_{(o)}^{n,1}$ and $p\ge1$, the $(L_p,Q)$-projection body of $K$, denoted by $\Pi_{p,Q}K\in\mathcal{K}_{(o)}^{n,m}$, is defined via its support function (see \cite[Definition 1.2]{Haddad-ye-lp-2025}):
\begin{align}\label{support function of ho lp projection body}
h_{\Pi_{p,Q} K}(\x)^p=\int_{S^{n-1}} h_Q\left(v^{\mathrm{T}}{}_{\bigcdot} \x\right)^p dS_{K, p}(v)\ \ \mathrm{for}\ \ \x \in \M.  \end{align}
Here, $S_{K,p}$ represents the $L_p$ surface area measure of $K$ such that $d S_{K, p}(u)=h_K^{1-p}(u) d S_K(u)$ for $u\in S^{n-1}$. If $K=B_2^n$, the unit ball of $\Mn$, then $dS_{K, p}(v)=dv$.  

\begin{proposition}\label{monotonicity and homogeneity}
Let $p\in[1,n)$, $Q\in \mathcal{K}_o^{1,m}$ and $K$ be a compact subset of $\Mn$. Then, the following statements hold.  

\vskip 1mm \noindent i) Monotonicity:  $C_{p,Q}(K)\leq C_{p,Q}(L)$, if $L\subset \Mn$ is compact such that $K\subset L$.

\vskip 1mm \noindent ii) 
Homogeneity: $C_{p, bQ}(aK)=b^pa^{n-p} C_{p, Q}(K)$, where $a,b>0$ and $a K=\{ax: x \in K\}$.

\vskip 1mm \noindent iii) Finiteness:   $C_{p,Q}(K)<\infty$.
 \end{proposition}

\begin{proof}
i) Let $K\subset L$ which gives $\mathbf{1}_L\ge \mathbf{1}_K$. It follows from \eqref{ek} that 
\begin{align*}
\mathscr{A}(L)=\Big\{f \in W_0^{1, p}: f \geq \mathbf{1}_L\Big\}\subset \Big\{f \in W_0^{1, p}: f \geq \mathbf{1}_K\Big\}=\mathscr{A}(K).
\end{align*}
This together with Definition \ref{def of cpqk with ek} yields $C_{p,Q}(K)\leq C_{p,Q}(L)$.

\vskip 2mm \noindent  ii) For $a>0$, let $g \in \mathscr{A}(a K)$, i.e., $ g\in W_0^{1,p}$ and $g\ge \mathbf{1}_{aK}.$  Let  $f\in W_0^{1,p}$  be given by $f(y)=g(a y)$ for  $y\in \Mn$. Thus,   $$f(y)=g(a y)\ge \mathbf{1}_{aK}(ay)=\mathbf{1}_{K}(y) \ \ \mathrm{for}\ \ y\in \Mn.$$
It follows from  \eqref{ek}  that $f\in \mathscr{A}(K)$.  The chain rule implies that $\nabla f(y)=a\nabla g(ay)$ and hence $\nabla_{\u} f(y)=a\nabla_{\u} g(ay)$ for $\u\in S^{nm-1}$. Combining with \eqref{hom-1-1}, \eqref{lp-norm} and Definition \ref{def of cpqk with ek}, one has
\begin{align*}
 C_{p,Q}(K)&=\inf_{f\in\mathscr{A}(K)}\Big(\int_{S^{nm-1}}\left\|h_Q\left(\nabla_{\u} f\right)\right\|_p^{-nm} d \u\Big)^{-\frac{p}{nm}}\\
 &=\inf_{f\in\mathscr{A}(K)}\Big(\int_{S^{nm-1}}\!\!\Big(\int_{\Mn}h_Q\big(\nabla_{\u} f(y)\big)^p d y\Big)^{-\frac{nm}{p}} \!\!d \u\Big)^{-\frac{p}{nm}} \\
&=\inf_{g\in\mathscr{A}(a K)}\Big(\int_{S^{nm-1}}\!\!\Big(\int_{\Mn}\Big(b^{-1}h_{bQ}\big(a\nabla_{\u} g(ay)\big)\Big)^p d y\Big)^{-\frac{nm}{p}} \!\!d \u\Big)^{-\frac{p}{nm}}\\
&=\frac{a^{p-n}}{b^{p}}\inf_{g\in\mathscr{A}(a K)}\Big(\int_{S^{nm-1}}\!\!\Big(\int_{\Mn}\Big(h_{bQ}\big(\nabla_{\u} g(x)\big)\Big)^p dx\Big)^{-\frac{nm}{p}} \!\!d\u\Big)^{-\frac{p}{nm}}\\&
=\frac{a^{p-n}}{b^{p}}C_{p,bQ}(a K),
\end{align*}
where we have used $ay=x$ (and hence $dy=a^{-n}dx$) in the second last equality.

\vskip 2mm \noindent  iii) To show the finiteness of $C_{p, Q}(\bigcdot)$, it is enough to prove   $C_{p, Q}(B_2^n)<\infty$, due to  the monotonicity and homogeneity of $C_{p, Q}(\bigcdot)$. For $\varepsilon>0$, let $f_{\varepsilon}\in W_0^{1, p}$ be defined by $f_{\varepsilon}(x)=1$ for $x\in B_2^n$,   $f_{\varepsilon}(x)=0$ for $|x|\geq 1+\varepsilon$, and 
\begin{align*}
f_{\varepsilon}(x)=1-\frac{|x|-1}{\varepsilon}\quad \mathrm{for} \quad|x|\in (1,1+\varepsilon).
\end{align*}
It follows from \eqref{ek} that $f_{\varepsilon}\in \mathcal{A}(B_2^n)$. A direct computation shows that
\begin{align}\label{nabla-f-ne-0}
\nabla f_{\varepsilon}(x)=-\frac{1}{\varepsilon}\nabla |x|=-\frac{x}{\varepsilon|x|}\ \ \mathrm{for}\ \ |x| \in(1,1+\varepsilon).
\end{align}
Moreover, as $f_{\varepsilon}\in W_0^{1, p}$ and is constant in the regions $|x|\leq 1$ and $|x|\ge1+\varepsilon$, by \cite[Corollary 1.21]{Heinonen 2006},  $\nabla f_{\varepsilon}(x)=0$ almost everywhere in these regions.  
This, together with \eqref{hom-1-1}, \eqref{lp-norm}, \eqref{support function of ho lp projection body}, \eqref{nabla-f-ne-0} and Fubini's theorem, yields that 
\begin{align*}
\left\|h_Q\big(\nabla_{\u} f_{\varepsilon}\big)\right\|_p^p 
&=\int_{\Mn}h_Q\big(\nabla_{\u} f_{\varepsilon}(x)\big)^p d x\\& =\int_{\big\{x: 1<|x|<1+\varepsilon\big\}}h_Q\bigg(\!\Big(-\frac{x}{\varepsilon|x|}\Big)^{\mathrm{T}}{}_{\bigcdot}{\u}\bigg)^p d x \\
&=\varepsilon^{-p} \int_1^{1+\varepsilon} r^{n-1} d r\int_{S^{n-1}}h_Q\big(-v^{\mathrm{T}}{}_{\bigcdot}{\u}\big)^p dv\\&=\frac{(1+\varepsilon)^n-1}{n\varepsilon^p} h_{\Pi_{p,Q} B_2^n}({\u})^p,
\end{align*}
where we have used the polar coordinate system in the third equality and the rotational invariance of $dv$ in the last equality.
Combining with Definition \ref{def of cpqk with ek}, \eqref{VnmKstar} and the fact that $\Pi_{p,Q}^* B_2^n\in \mathcal{K}_{(o)}^{n,m}$, one gets that, for $p\in[1,n)$,
\begin{align}\label{CpQB less than infty}
\nonumber C_{p, Q}(B_2^n) &\leq\left.\Big(\int_{S^{nm-1}}\left\|h_Q\left(\nabla_{\u} f_{\varepsilon}\right)\right\|_p^{-nm} d {\u}\Big)^{-\frac{p}{nm}}\right|_{\varepsilon=1}\\ 
\nonumber&= \frac{(1+\varepsilon)^n-1}{n\varepsilon^p}\bigg|_{\varepsilon=1}\Big(\int_{S^{nm-1}}h_{\Pi_{p,Q} B_2^n}({\u})^{-nm}
d {\u}\Big)^{-\frac{p}{nm}} \\
&<2^{n}\left(nmV_{nm}(\Pi_{p,Q}^* B_2^n)\right)^{-\frac{p}{nm}}<\infty.
\end{align} This completes the proof.
\end{proof}

Note that $C_{p,Q}(\bigcdot)$ can also be defined when $p\in(0,1)$ and $p\in[n,\infty)$, but those values are always $0$. Specifically, when $p\in (0,1)$ and $\varepsilon\to 0^+$, it follows from \eqref{CpQB less than infty} that $C_{p,Q}(B_2^n)=0$. Hence,  $C_{p,Q}(K)=0$ for any compact set $K\subset \Mn$, due to the monotonicity and homogeneity of $C_{p,Q}(\bigcdot)$. When $p\in[n,\infty)$, it follows from Theorem \ref{cpqk and cpk} and \eqref{cpb2n} that $C_{p,Q}(K)=0$.

We now provide two equivalent definitions of $C_{p, Q}(\bigcdot)$.  Let $f^+=\max\{f,0\}$. Note that by \cite[Lemma 1.19]{Heinonen 2006}, for any $f \in W_0^{1, p}$, one has $f^+ \in W_0^{1, p}$,
\begin{align}\label{nabla f plus and nabla f}
\nabla f^{+}(x)= \nabla f(x)\ \ \mathrm{ 
 for } \ f(x)>0\quad\mathrm{and}\quad\nabla f^{+}(x)=0 \ \mathrm{for} \ f(x)\leq 0.
\end{align} We shall need the following well-known fact: there exist   $\phi_j\in C_c^{\infty}$, $j\in\mathbb{N}$,    such that $\phi_j(t)=0$ for $t\in (-\infty, 0]$,   $\phi_j(t)=1$ for $t\in [1, \infty)$, and $
0 \leq \phi_j^{\prime}(t) \leq j^{-1}+1$ for $t\in \mathbb{R}.$  
\begin{theorem}\label{equiv def of cpqk of ek}
Let $p\in[1,n)$, $Q\in \mathcal{K}_o^{1,m}$ and $K$ be a compact subset of $\Mn$. The $m$th order $p$-affine capacity can be equivalently formulated as follows.
\vskip 1mm \noindent i) 
Let $\mathscr{B}(K)$ be the set of $f \in W_0^{1, p}$ such that $f \geq 1$ on $K$. Then
\begin{align}\label{CPQ=BK}
C_{p, Q}(K)=\inf_{f \in \mathscr{B}(K)} \Big(\int_{S^{nm-1}}\left\|h_Q\left(\nabla_{\u} f\right)\right\|_p^{-nm} d {\u}\Big)^{-\frac{p}{nm}}.
\end{align}
\vskip 1mm \noindent ii) 
Let $\mathscr{D}(K)$ be the set of $f \in W_0^{1, p}$ such that $0 \leq f \leq 1$ on $\Mn$ and $f=1$ in a neighborhood of $K$. Then
\begin{align}\label{def for f in DK}
C_{p, Q}(K)=\inf_{f \in \mathscr{D}(K)} \Big(\int_{S^{nm-1}}\left\|h_Q\left(\nabla_{\u} f\right)\right\|_p^{-nm} d \u\Big)^{-\frac{p}{nm}}.
\end{align}
\end{theorem}

\begin{proof}
\vskip 1mm \noindent i) As $\mathscr{A}(K) \subset \mathscr{B}(K)$, we only need to show that  \eqref{CPQ=BK} holds with ``$=$" replaced by ``$\leq $". Let $\left\{f_j\right\}_{j\in \mathbb{N}}\subset\mathscr{B}(K)$ be a sequence such that
\begin{align}\label{DF 2}
\lim _{j \to\infty} \Big(\int_{S^{nm-1}}\left\|h_Q\left(\nabla_{\u} f_j\right)\right\|_p^{-nm} d \u\Big)^{-\frac{p}{nm}}\!\!=\!\inf _{f \in \mathscr{B}(K)} \Big(\int_{S^{nm-1}}\left\|h_Q\left(\nabla_{\u} f\right)\right\|_p^{-nm} d \u\Big)^{-\frac{p}{nm}}.
\end{align}
It can be easily checked that $\big\{f_j^+\big\}_{j\in\mathbb{N}}\subset\mathscr{A}(K)$. Combining with \eqref{nabla f plus and nabla f}, it yields that 
\begin{align*}
h_Q\big(\nabla_{\u} f_j^{+}\big)= h_Q\big(\nabla_{\u} f_j\big) \ \ \mathrm{for}~f_j>0\ \ \mathrm{and}\ \ h_Q\big(\nabla_{\u} f_j^{+}\big)=0\ \ \mathrm{for}\ \ f_j\leq 0,    
\end{align*}
 for any $j\in\mathbb{N}$ and $\u \in S^{nm-1}$. Thus, 
$h_Q\big(\nabla_{\u} f_j^{+}\big) \leq h_Q\big(\nabla_{\u} f_j\big)$ and  
\begin{align*}
\Big(\int_{S^{nm-1}}\big\|h_Q(\nabla_{\u} f_j^{+})\big\|_p^{-nm} d \u\Big)^{-\frac{p}{nm}}\leq\Big(\int_{S^{nm-1}}\big\|h_Q(\nabla_{\u} f_j)\big\|_p^{-nm} d \u\Big)^{-\frac{p}{nm}}.  
\end{align*}
This, together with Definition \ref{def of cpqk with ek} and \eqref{DF 2}, implies that
\begin{align*}
C_{p, Q}(K)&=\!\!\!\!\!\inf _{f \in \mathscr{A}(K)}\!\!\Big(\!\int_{S^{nm-1}}\!\left\|h_Q(\nabla_{\u} f)\right\|_p^{-nm} d \u\Big)^{-\frac{p}{nm}}
\\ &\leq \limsup _{j\to\infty}\Big(\int_{S^{nm-1}}\left\|h_Q(\nabla_{\u} f_j^+)\right\|_p^{-nm} d \u\Big)^{-\frac{p}{nm}}\\
&\leq \!\!\lim _{j \to\infty}\!\Big(\int_{S^{nm-1}}\left\|h_Q(\nabla_{\u} f_j)\right\|_p^{-nm} d \u\Big)^{-\frac{p}{nm}}\\ &=\inf _{f \in \mathscr{B}(K)} \Big(\int_{S^{nm-1}}\left\|h_Q(\nabla_{\u} f)\right\|_p^{-nm} d \u\Big)^{-\frac{p}{nm}}.
\end{align*}
Thus,  the desired result in \eqref{CPQ=BK} holds.

\vskip 2mm \noindent ii) Notice that $\mathscr{D}(K)\subset\mathscr{A}(K)$. Thus, we only need to show that \eqref{def for f in DK} holds  with ``$=$" replaced by ``$ \geq$". Let $f\in \mathscr{A}(K)$, i.e., $f\in W_0^{1,p}$ and $f\ge \mathbf{1}_K$. It follows from  \cite[Theorem 1.18]{Heinonen 2006} that $\nabla \phi_j(f)=\phi_j^{\prime}\nabla f$. Thus, $\nabla_{\u}\phi_j(f)=\phi_j^{\prime}\nabla_{\u}f$ for any $\u\in S^{nm-1}$. Moreover,  $\phi_j(f) \in \mathscr{D}(K)$ for any $j\in \mathbb{N}$.   Together with \eqref{hom-1-1} and \eqref{lp-norm}, one gets
\begin{align*}
 \inf_{g \in \mathscr{D}(K)}\Big(\int_{S^{nm-1}}\left\|h_Q\big(\nabla_{\u} g\big)\right\|_p^{-nm} d \u\Big)^{-\frac{p}{nm}}
 &\leq \Big(\int_{S^{nm-1}}\left\|h_Q\big(\nabla_{\u} \phi_j(f)\big)\right\|_p^{-nm} d \u\Big)^{-\frac{p}{nm}}\\
 &=\Big(\int_{S^{nm-1}}\left\|h_Q\big(\phi_j^{\prime}\nabla_{\u} f\big)\right\|_p^{-nm} d \u\Big)^{-\frac{p}{nm}}\\
 &\leq (j^{-1}+1)^p \Big(\int_{S^{nm-1}}\left\|h_Q\left(\nabla_{\u} f\right)\right\|_p^{-nm} d \u\Big)^{-\frac{p}{nm}}.
\end{align*} The desired result in \eqref{def for f in DK} is an immediate consequence from Definition \ref{def of cpqk with ek}, after taking the limit as $j \to\infty$ on  both sides of the above inequality, and the infimum over $f\in \mathscr{A}(K)$.
\end{proof}

Replacing $W_0^{1,p}$ by $C_c^{\infty}$ in   Definition \ref{def of cpqk with ek} and Theorem \ref{equiv def of cpqk of ek}  will be more convenient in later context. We shall do this in the next corollary.  
\begin{corollary}\label{cpqk-f-smooth}
Let $p \in[1, n)$, $Q\in\mathcal{K}_o^{1,m}$ and $K$ be a compact set in $\Mn$. Then, 
\begin{align}\label{Theorem-3-4-formula}
C_{p, Q}(K)=\inf_{f \in C_c^{\infty} \cap \mathcal{E}}\Big(\int_{S^{nm-1}}\left\|h_Q\left(\nabla_{\u} f\right)\right\|_p^{-nm} d \u\Big)^{-\frac{p}{nm}},
\end{align} where $\mathcal{E}$ can be either of $\mathscr{A}(K), \mathscr{B}(K)$ and $  \mathscr{D}(K)$.
\end{corollary}
\begin{proof} As $ \mathscr{D}(K)\subset   \mathscr{A}(K)\subset \mathscr{B}(K),$ it follows that 
\begin{align}\label{inclusion-C0infty}
C_c^{\infty} \cap \mathscr{D}(K)\subset C_c^{\infty} \cap \mathscr{A}(K)\subset C_c^{\infty} \cap \mathscr{B}(K).
\end{align}
We will prove \eqref{Theorem-3-4-formula} for $f\in C_c^{\infty} \cap \mathscr{B}(K)$ first and  then  for $f \in C_c^{\infty} \cap \mathscr{D}(K)$.

If $\mathcal{E}=\mathscr{B}(K)$, to prove \eqref{Theorem-3-4-formula} is equivalent to show \begin{align}\label{CPQ-ge-bk}
C_{p, Q}(K)\ge \inf_{f \in C_c^{\infty} \cap \mathscr{B}(K)}\Big(\int_{S^{nm-1}}\left\|h_Q(\nabla_{\u} f)\right\|_p^{-nm} d \u\Big)^{-\frac{p}{nm}},
\end{align} due to  $C_c^{\infty}\!\subset\! W_0^{1,p}$ and Theorem \ref{equiv def of cpqk of ek}. To this end, let $f\in \mathscr{D}(K)$. According to the proof of \cite[Theorem 3.3]{Hong and Ye 2018},  there exists a sequence $\{f_j\}_{j\in\mathbb{N}}\subset C_c^{\infty} \cap \mathscr{B}(K)$ such that $$f_j\rightarrow f\ \  \mathrm{in}\ \ W_0^{1,p}, \ \ \mathrm{i.e.,}\ \  \|f_j-f\|_{1,p}\rightarrow 0  \ \ \mathrm{as}\ \  j\rightarrow \infty.$$ 
Moreover, as $Q\in\mathcal{K}_o^{1,m}$, there exists $\varrho>0$ such that $Q\subset \varrho B_2^m$. By \cite[Lemma 1.8.12]{schneider}, one gets 
\begin{align*}
|h_Q(x)-h_Q(y)|\leq \varrho|x-y| \ \ \mathrm{for\ any}\ x,y \in \Mm.
\end{align*}
Combining with the triangle inequality, one gets, for any $p\in [1,n)$ and $\u\in S^{nm-1}$,
\begin{align}\label{p norm of nabla u fj}
\left|\left\|h_Q\left(\nabla_{\u} f_j\right)\right\|_p-\left\|h_Q\left(\nabla_{\u} f\right)\right\|_p\right|
\leq
\left\|h_Q\left(\nabla_{\u} f_j\right)-h_Q\left(\nabla_{\u} f\right)\right\|_p
\leq 
\varrho\left\|\nabla_{\u} f_j-\nabla_{\u} f\right\|_p.
\end{align}
As $\u=(u_1,\cdots,u_m)\in S^{nm-1}$ with $u_i\in \Mn$, $i=1,\cdots, m$, one has
\begin{align*}
|\nabla_{\u} f_j\!-\!\nabla_{\u} f|\!=\!|(\nabla f_j\!-\!\nabla f)^{\mathrm{T}}{}_{\bigcdot}\u|\!=\!|((\nabla f_j\!-\!\nabla f)\bigcdot u_1,\cdots,(\nabla f_j\!-\!\nabla f)\bigcdot u_m)|\!\leq \!|\nabla f_j\!-\!\nabla f|.
\end{align*}
Together with \eqref{p norm of nabla u fj},  one has, 
\begin{align*}
\left|\left\|h_Q\left(\nabla_{\u} f_j\right)\right\|_p-\left\|h_Q\left(\nabla_{\u} f\right)\right\|_p\right|\leq \varrho\left\|\nabla f_j-\nabla f\right\|_p.
\end{align*}
As $f_j\rightarrow f$ in $W_0^{1,p}$, we obtain that, for any $\u\in S^{nm-1}$,
$$\lim_{j\to\infty} \left\|h_Q\left(\nabla_{\u} f_j\right)\right\|_p=\left\|h_Q\left(\nabla_{\u} f\right)\right\|_p.$$
 It follows from  Fatou's lemma and   $f_j\in C_c^{\infty} \cap \mathscr{B}(K)$  that, for any $f\in \mathscr{D}(K)$,
\begin{align*}
\Big(\int_{S^{nm-1}}\left\|h_Q\left(\nabla_{\u} f\right)\right\|_p^{-nm} d \u\Big)^{-\frac{p}{nm}}
& =\Big(\int_{S^{nm-1}} \lim _{j \to \infty}\left\|h_Q\left(\nabla_{\u} f_j\right)\right\|_p^{-nm} d \u\Big)^{-\frac{p}{nm}} \\
& \geq\left(\liminf _{j \to \infty} \int_{S^{nm-1}}\left\|h_Q\left(\nabla_{\u} f_j\right)\right\|_p^{-nm} d \u\right)^{-\frac{p}{nm}} \\
& =\limsup _{j \to \infty}\Big(\int_{S^{nm-1}}\left\|h_Q\left(\nabla_{\u} f_j\right)\right\|_p^{-nm} d \u\Big)^{-\frac{p}{nm}}\\
&\ge 
\inf _{g \in C_c^{\infty} \cap \mathscr{B}(K)}\Big(\int_{S^{nm-1}}\left\|h_Q\left(\nabla_{\u} g\right)\right\|_p^{-nm} d \u\Big)^{-\frac{p}{nm}}.
\end{align*}Consequently,  \eqref{CPQ-ge-bk} follows from Theorem \ref{equiv def of cpqk of ek}, after taking the infimum over $f\in \mathscr{D}(K)$.   This completes the proof for $\mathcal{E}=\mathscr{B}(K).$

Now let us consider  $\mathcal{E}=\mathscr{D}(K)$. To this end, let $g\in C_c^{\infty} \cap \mathscr{B}(K)$. It can be easily checked that $\phi_j(g) \in C_c^{\infty}\cap\mathscr{D}(K)$, where $\phi_j$ is given  above Theorem \ref{equiv def of cpqk of ek}. 
Together with  \eqref{hom-1-1}, \eqref{lp-norm} and  $\nabla_{\u}\phi_j(g)=\phi_j^{\prime}\nabla_{\u}g$, one has, 
\begin{align*}
 \inf_{f \in C_c^{\infty} \cap\mathscr{D}(K)}\bigg(\int_{S^{nm-1}}\left\|h_Q\left(\nabla_{\u} f\right)\right\|_p^{-nm}d \u\bigg)^{-\frac{p}{nm}}
 &\leq \bigg(\int_{S^{nm-1}}\left\|h_Q\big(\nabla_{\u} \phi_j(g)\big)\right\|_p^{-nm} d \u\bigg)^{-\frac{p}{nm}}\\
 &\leq (j^{-1}+1)^p \bigg(\int_{S^{nm-1}}\left\|h_Q\left(\nabla_{\u} g\right)\right\|_p^{-nm} d \u\bigg)^{-\frac{p}{nm}}.
\end{align*}
By taking $j \to\infty$, one gets \begin{align*}
 \inf_{f \in C_c^{\infty} \cap\mathscr{D}(K)}\bigg(\int_{S^{nm-1}}\left\|h_Q\left(\nabla_{\u} f\right)\right\|_p^{-nm}d \u\bigg)^{-\frac{p}{nm}}
  \leq   \bigg(\int_{S^{nm-1}}\left\|h_Q\left(\nabla_{\u} g\right)\right\|_p^{-nm} d \u\bigg)^{-\frac{p}{nm}}.
\end{align*} After taking the infimum over $g\in C_c^{\infty} \cap\mathscr{B}(K)$, the following holds: 
\begin{align*} C_{p,Q}(K) &=\inf_{g\in C_c^{\infty} \cap \mathscr{B}(K)}\Big(\int_{S^{nm-1}}\left\|h_Q\left(\nabla_{\u} g\right)\right\|_p^{-nm} d \u\Big)^{-\frac{p}{nm}}\\&\geq \inf_{f \in C_c^{\infty} \cap \mathscr{D}(K)}\Big(\int_{S^{nm-1}}\left\|h_Q(\nabla_{\u} f)\right\|_p^{-nm} d \u\Big)^{-\frac{p}{nm}}\\&\geq \inf_{f \in C_c^{\infty} \cap \mathscr{B}(K)}\Big(\int_{S^{nm-1}}\left\|h_Q(\nabla_{\u} f)\right\|_p^{-nm} d \u\Big)^{-\frac{p}{nm}}=C_{p,Q}(K),
\end{align*} where the last inequality follows from $\mathscr{D}(K)\subset \mathscr{B}(K)$. Thus, \eqref{Theorem-3-4-formula} holds for $\mathcal{E}=\mathscr{D}(K)$. By \eqref{inclusion-C0infty},  \eqref{Theorem-3-4-formula}  for $\mathcal{E}=\mathscr{A}(K)$ follows immediately, and this completes the proof. 
\end{proof}

We now prove some other fundamental properties for $C_{p,Q}(\bigcdot)$, including for example, the symmetry and concavity with respect to $Q\in\mathcal{K}_o^{1,m}$, sub-additivity, translation invariance,  affine invariance  and upper semi-continuity, among others.

Let $Q_1, Q_2\in\mathcal{K}_{o}^{1,m}$, and $p \geq 1$. The $L_p$ sum of $Q_1$ and $Q_2$, denoted by $\lambda\bigcdot_pQ_1+_p \mu\bigcdot_pQ_2$, for $\lambda\ge0$ and $\mu \geq 0$, is defined via the support function
\begin{align}\label{lp-addition-with-lp-multi}
h_{\lambda \bigcdot_p Q_1+{ }_p \mu \bigcdot_p Q_2}(u)^p=\lambda h_{Q_1}(u)^p+\mu h_{Q_2}(u)^p,\ \ u\in S^{m-1}.
\end{align}
Here, the $p$-scalar multiplication is given by $\lambda \bigcdot{ }_p Q_1:=\lambda^{1/p} Q_1$. Denote by $GL(n)$   the general linear group of $n\times n$ invertible matrices. Let $\det (\phi)$ be the determinant of $\phi\in GL(n)$.  For $\phi\in GL(n)$ and $\u=(u_1,\cdots,u_m)\in \M$ with $u_i\in \Mn$, $i=1,\cdots,m$, let  
\begin{align}\label{bar T u}
\phi{}_{\bigcdot}\u=(\phi{}_{\bigcdot}u_1,\cdots,\phi{}_{\bigcdot}u_m).
\end{align}

\begin{proposition}\label{proposition 1}
Let $p\in[1,n)$, $Q,Q_1, Q_2\in \mathcal{K}_{o}^{1,m}$, and $K, K_1, K_2$ be compact sets in $ \Mn$. The following statements hold.
\vskip 1mm \noindent i) 
Symmetry: $C_{p, Q}(K)=C_{p,-Q}(K)$.
\vskip 1mm \noindent ii)  Concavity : for any $\lambda\in [0,1]$, one has
$$
C_{p, \lambda\bigcdot_{p} Q_1+_{p}(1-\lambda) \bigcdot_{p}Q_2}(K) \geq \lambda C_{p, Q_1}(K)+(1-\lambda)C_{p, Q_2}(K).
$$
\vskip 1mm \noindent iii) 
Sub-additivity: $C_{p,Q}(K_1\cup K_2)\leq C_{p,Q}(K_1)+C_{p,Q}(K_2).$
\vskip 1mm \noindent iv) 
Translation invariance: $C_{p, Q}(K+z)=C_{p, Q}(K)$, for any $z\in \Mn$.
\vskip 1mm \noindent v) 
The $m$th order boundary capacity: $C_{p, Q}(K)=C_{p, Q}(\partial K)
$.
\vskip 1mm \noindent vi) 
Affine invariance: 
$C_{p, Q}(\phi{}_{\bigcdot}K)=|\det(\phi)|^{\frac{n-p}{n}} C_{p, Q}(K)$, for $\phi\in GL(n)$. 
\end{proposition}
\begin{proof}
\vskip 1mm \noindent i) It follows from $\nabla_{\u} f=\nabla f^{\mathrm{T}}{}_{\bigcdot}\u$, \eqref{h KB} and Definition \ref{def of cpqk with ek} that
\begin{align*}
C_{p,-Q}(K)&=\inf_{f\in\mathscr{A}(K)}\Big(\int_{S^{nm-1}}\left\|h_{-Q}\left(\nabla f^{\mathrm{T}}{}_{\bigcdot}\u\right)\right\|_p^{-nm} d\u\Big)^{-\frac{p}{nm}}\\ 
&= \inf_{f\in\mathscr{A}(K)}\Big(\int_{S^{nm-1}}\left\|h_{-Q{}_{\bigcdot}\u^{\mathrm{T}}}\left(\nabla f\right)\right\|_p^{-nm}d\u\Big)^{-\frac{p}{nm}}\\
&=\inf_{f\in\mathscr{A}(K)}\Big(\int_{S^{nm-1}}\left\|h_{Q{}_{\bigcdot}{(-\u)}^{\mathrm{T}}}(\nabla f)\right\|_p^{-nm}d\u\Big)^{-\frac{p}{nm}}\\
&
=\inf_{f\in\mathscr{A}(K)}\Big(\int_{S^{nm-1}}\left\|h_{Q}\big(\nabla f^{\mathrm{T}}{}_{\bigcdot}{(-\u)}\big)\right\|_p^{-nm}d{\u}\Big)^{-\frac{p}{nm}}\\
&=\inf_{f\in\mathscr{A}(K)}\Big(\int_{S^{nm-1}}\left\|h_{Q}\big(\nabla f^{\mathrm{T}}{}_{\bigcdot}\u\big)\right\|_p^{-nm} d{\u}\Big)^{-\frac{p}{nm}}=C_{p,Q}(K),
\end{align*}
where we have used the rotational invariance of $d\u$ in the second last equation.

\vskip 1mm \noindent ii) It follows from \eqref{lp-addition-with-lp-multi} that  
\begin{align*}
\left\|h_{\lambda\bigcdot_{p} Q_1+_{p}(1-\lambda) \bigcdot_{p}Q_2}\left(\nabla_{\u} f\right)\right\|_p^{p}
 &=\int_{\Mn} h_{\lambda\bigcdot_{p} Q_1+_{p}(1-\lambda) \bigcdot_{p}Q_2}\big(\nabla_{\u} f(x)\big)^p dx \\
 &=\int_{\Mn} 
 \Big(\lambda h_{Q_1}\big(\nabla_{\u} f(x)\big)^p+(1-\lambda)h_{Q_2}\big(\nabla_{\u} f(x)\big)^p
 \Big) dx\\
&=\lambda \left\|h_{Q_1}\left(\nabla_{\u} f\right)\right\|_p^{p}+(1-\lambda)\left\|h_{Q_2}\left(\nabla_{\u} f\right)\right\|_p^{p}.
\end{align*}
Together with the reverse Minkowski inequality, one has, 
\begin{align*}
&\Big(\int_{S^{nm-1}}\left\|h_{\lambda\bigcdot_{p} Q_1+_{p}(1-\lambda) \bigcdot_{p}Q_2}\left(\nabla_{\u} f\right)\right\|_p^{-nm} d \u\Big)^{-\frac{p}{nm}}\\  
&=\Big(\int_{S^{nm-1}}\Big(\lambda\left\|h_{Q_1}\left(\nabla_{\u} f\right)\right\|_p^{p}+(1-\lambda)\left\|h_{Q_2}\left(\nabla_{\u} f\right)\right\|_p^{p} \Big)^{\frac{-nm}{p}}d \u\Big)^{-\frac{p}{nm}}\\
&\ge\lambda\Big(\int_{S^{nm-1}}\left\|h_{Q_1}\left(\nabla_{\u} f\right)\right\|_p^{-nm}d \u\Big)^{-\frac{p}{nm}}+(1-\lambda)\Big(\int_{S^{nm-1}}\left\|h_{Q_2}\left(\nabla_{\u} f\right)\right\|_p^{-nm}d \u\Big)^{-\frac{p}{nm}}.
\end{align*} The desired concavity is an immediate consequence after taking  the infimum over $f\in\mathscr{A}(K)$.

\vskip 2mm \noindent iii) Let $f_1\in\mathscr{A}(K_1)$ and $f_2\in \mathscr{A}(K_2)$. It follows from \eqref{ek} that $f_1\ge \mathbf{1}_{K_1}$ and $f_2\ge \mathbf{1}_{K_2}$. Then, $$f_1+f_2\ge \mathbf{1}_{K_1}+\mathbf{1}_{K_2}\ge\mathbf{1}_{K_1\cup K_2},$$ which implies $f_1+f_2\in \mathscr{A}(K_1\cup K_2)$. 
Therefore, $$\mathscr{A}(K_1)+\mathscr{A}(K_2)\subset \mathscr{A}(K_1\cup K_2).$$ The sub-additivity follows immediately from  Definition \ref{def of cpqk with ek}.

\vskip 2mm \noindent iv) For $z\in \Mn$, let $f \in \mathscr{A}(K+z)$, i.e., $f\in W_0^{1,p}$ and $f\ge \mathbf{1}_{K+z}$. Let $g(x)=f(x+z)$ for $x\in\Mn$. Then, $\nabla g(x)=\nabla f(x+z)$ and 
\begin{align*}
g(x)=f(x+z)\ge \mathbf{1}_{K+z}(x+z)=\mathbf{1}_{K}(x).
\end{align*}
It follows from \eqref{ek} that $g\in\mathscr{A}(K)$. In particular, one gets 
\begin{align*}
\left\|h_Q(\nabla_{\u} g)\right\|_p^p&=\int_{\Mn}h_Q\big(\nabla g(x)^{\mathrm{T}}{}_{\bigcdot}\u\big)^pdx=\int_{\Mn}h_Q\big(\nabla f(x+z)^{\mathrm{T}}{}_{\bigcdot}\u\big)^pdx\\
&=\int_{\Mn}h_Q\big(\nabla f(y)^{\mathrm{T}}{}_{\bigcdot}\u\big)^pdy=\left\|h_Q(\nabla_{\u} f)\right\|_p^p,
\end{align*}
where the second last equality is obtained by letting $y=x+z$ and hence $dx=dy$.
The desired translation invariance is an immediate consequence of Definition \ref{def of cpqk with ek}.

\vskip 2mm \noindent v) It follows from $\partial K\subset K$ and Proposition \ref{monotonicity and homogeneity} (i) that $C_{p, Q}(\partial K) \leq C_{p, Q}(K)$. Hence, the desired result follows if  $C_{p, Q}(K)\leq C_{p, Q}(\partial K)$. To this end, for any $f \in \mathscr{A}(\partial K)$,  let $$g=\max \Big\{f, 1\Big\} \ \ \mathrm{on} \  \ K  \ \ \mathrm{and} \ \  g=f\ \  \mathrm{on} \ \ \Mn \setminus K.$$ Thus, $g \in \mathscr{A}(K)$. We now claim that
\begin{align}\label{nabla-g-x}
\nabla g(x)=o \ \ \mathrm{if}\ x\in K \   \mathrm{such\ that}\ \ g(x)=1,\ \ \mathrm{and} \ \ \nabla g(x)=\nabla f(x)\ \ \mathrm{otherwise}.
\end{align}  Indeed,  the function $g$ on $K$ can be rewritten as
$$
g(x)=1+(f(x)-1)^{+}. 
$$ On the open set $ \Mn\setminus K$, $g=f$ and hence $\nabla g=\nabla f.$  By \eqref{nabla f plus and nabla f}, one gets   
\begin{align*}
\nabla (f-1)^+(x)&=\nabla f(x) \ \ \mathrm{if}\  x\in K \  \mathrm{such\ that}\ \ f(x)>1 \  \mathrm{or\ equivalently}\ g(x)=f(x),\\   \nabla (f-1)^+(x)&=o \ \ \ \ \ \ \ \ \   \mathrm{if}\  x\in K \  \mathrm{such\ that}\ \ f(x)\leq 1 \ \mathrm{or\ equivalently}\ g(x)=1. 
\end{align*}
This shows \eqref{nabla-g-x}. Consequently, $h_Q(\nabla_{\u} g(y))\leq h_Q(\nabla_{\u} f(y))$ for any $\u\in S^{nm-1}$ and $y\in\Mn$. It follows from Definition \ref{def of cpqk with ek} that 
\begin{align*}
C_{p, Q}(K)&=\inf_{g\in\mathscr{A}(K)}\Big(\int_{S^{nm-1}}\Big(\int_{\Mn}h_Q\big(\nabla_{\u} g(y)\big)^p d y\Big)^{-\frac{nm}{p}} d \u\Big)^{-\frac{p}{nm}}\\
&\leq \inf_{f\in\mathscr{A}(\partial K)}\Big(\int_{S^{nm-1}}\Big(\int_{\Mn}h_Q\big(\nabla_{\u} f(y)\big)^p d y\Big)^{-\frac{nm}{p}} d \u\Big)^{-\frac{p}{nm}}= C_{p, Q}(\partial K).
\end{align*}

\vskip 2mm \noindent vi) Let $\phi\in GL(n)$ and  $f \in \mathscr{A}(\phi{}_{\bigcdot}K)$, i.e., $f\in W_0^{1,p}$ and $f\ge \mathbf{1}_{\phi{}_{\bigcdot}K}$. Let  $g=f\circ \phi$, i.e., $g(y)=f(\phi{}_{\bigcdot}{y})$ for $y\in \Mn$. Thus, 
\begin{align*}
g(y)=f(\phi{}_{\bigcdot}y)\ge \mathbf{1}_{\phi{}_{\bigcdot}K}(\phi{}_{\bigcdot}y)=\mathbf{1}_{K}(y).
\end{align*}
It follows from \eqref{ek} that $g \in \mathscr{A}(K)$. Moreover, by the chain rule, one gets
\begin{align*}
\nabla g(y)=\nabla (f\circ \phi)(y)=\phi^{\mathrm{T}}{}_{\bigcdot} \nabla f(\phi{}_{\bigcdot}y).
\end{align*}

For $\u\in S^{nm-1}$ and $y\in \Mn$, let $x=\phi{}_{\bigcdot}y$ and $\v=\frac{\phi{}_{\bigcdot}\u}{|\phi{}_{\bigcdot}\u|}$. Thus, $dx=|\det (\phi)|dy$ and $d\v=|\phi{}_{\bigcdot}\u|^{-nm}|\det (\phi)|^md\u$. Together with Definition \ref{def of cpqk with ek} and equality above, one gets
\begin{align*}
C_{p,Q}(K)
=&\inf_{g\in\mathscr{A}(K)}\Big(\int_{S^{nm-1}}\Big(\int_{\Mn}h_Q\big(\nabla_{\u} g(y)\big)^p d y\Big)^{-\frac{nm}{p}} d\u\Big)^{-\frac{p}{nm}} \\
=&\inf_{f\in\mathscr{A}(\phi{}_{\bigcdot}K)}\Big(\int_{S^{nm-1}}\Big(\int_{\Mn}h_Q\Big(\left(\phi^{\mathrm{T}}{}_{\bigcdot} \nabla f(\phi{}_{\bigcdot}y)\right)^{\mathrm{T}}{}_{\bigcdot}\u\Big)^p d y\Big)^{-\frac{nm}{p}} d \u\Big)^{-\frac{p}{nm}} \\
=&\inf_{f\in\mathscr{A}(\phi{}_{\bigcdot}K)}\Big(\int_{S^{nm-1}}\Big(\int_{\Mn}h_Q\Big(\nabla f(\phi{}_{\bigcdot}y)^{\mathrm{T}}{}_{\bigcdot}\frac{\phi{}_{\bigcdot} \u}{|\phi{}_{\bigcdot} \u|}\Big)^p d y\Big)^{-\frac{nm}{p}} |\phi{}_{\bigcdot} \u|^{-nm}d \u\Big)^{-\frac{p}{nm}} \\
 =&\inf_{f\in\mathscr{A}(\phi{}_{\bigcdot}K)}\Big(\int_{S^{nm-1}}\Big(\int_{\Mn}h_Q\big(\nabla f(x)^{\mathrm{T}}{}_{\bigcdot}\v\big)^p dx\Big)^{-\frac{nm}{p}}|\det (\phi)|^{\frac{nm}{p}-m} d\v\Big)^{-\frac{p}{nm}}\\
=&|\det (\phi)|^{\frac{p-n}{n}}C_{p,Q}(\phi{}_{\bigcdot}K).
\end{align*}
The desired affine invariance is an immediate consequence.
\end{proof}

We now establish the upper semi-continuity and the continuity from above for the $m$th order $p$-affine capacity $C_{p,Q}(\bigcdot)$.
\begin{proposition}
Let $p\in[1,n)$, $Q\in \mathcal{K}_o^{1,m}$ and $K$ be a compact subset of $\Mn$. The following statements hold.
\vskip 1mm \noindent i)  Upper semi-continuity: for each $\varepsilon>0$, there is an open set $O$, such that, for every compact set $F$ with $K \subset F \subset O$, one gets
$$C_{p,Q}(F) \leq C_{p,Q}(K)+\varepsilon.$$
\vskip 1mm \noindent ii)  Continuity from above: if $\left\{K_j\right\}_{j\in\mathbb{N}}$ is a sequence of compact subsets in $\Mn$ such that $K_{j+1}\subset K_j$ for all $j\in \N$, then 
$$\lim _{j \to \infty} C_{p, Q}\left(K_j\right)=C_{p, Q}\left(\cap_{j=1}^{\infty} K_j\right).$$
\end{proposition}
\begin{proof}
\vskip 1mm \noindent i)  The desired upper semi-continuity is an immediate consequence of \eqref{def for f in DK}. More precisely, for $\varepsilon>0$, there exists  $f_{\varepsilon} \in\mathscr{D}(K)$, i.e., $0 \leq f_{\varepsilon} \leq 1$ on $\Mn$ and $f_{\varepsilon}=1$ in an open neighborhood of $K$ (say $O$), such that $K\subset O$ and 
\begin{align}\label{FK 1}
\Big(\int_{S^{nm-1}}\left\|h_Q\left(\nabla_{{\u}} f_{\varepsilon}\right)\right\|_p^{-nm} d\u\Big)^{-\frac{p}{nm}}\leq C_{p, Q}(K)+\varepsilon.
\end{align} For any compact set $F$ such that $K\subset F\subset O$, one has  $f_{\varepsilon}\in \mathscr{D}(F)$ as well. Hence, by \eqref{def for f in DK} and \eqref{FK 1}, one gets 
\begin{align*}
C_{p, Q}(F)\leq \Big(\int_{S^{nm-1}}\left\|h_Q\left(\nabla_{{\u}} f_{\varepsilon}\right)\right\|_p^{-nm} d\u\Big)^{-\frac{p}{nm}} \leq C_{p, Q}(K)+\varepsilon.
\end{align*}
Thus, the desired upper semi-continuity is obtained.

\vskip 1mm \noindent ii) First of all, the compactness of $\cap_{j=1}^{\infty} K_j$ is clear due to the facts that each $K_j$ is compact and $K_{j+1}\subset K_j$ for all $j\in \N$. Thus, $C_{p, Q}(\cap_{j=1}^{\infty} K_j)$ is well-defined and  
by Corollary \ref{cpqk-f-smooth}, for any $\varepsilon\in [0,1)$, there exists a function $f _{\varepsilon}\in C_c^{\infty}\cap \mathscr{A}(\cap_{j=1}^{\infty} K_j)$ such that
\begin{align}\label{kj plus 1 and kj}
\Big(\int_{S^{nm-1}}\left\|h_Q(\nabla_{\u} f _{\varepsilon})\right\|_p^{-nm} d\u\Big)^{-\frac{p}{nm}} \leq C_{p, Q}\left(\cap_{j=1}^{\infty} K_j\right)+\varepsilon.    
\end{align} It can be checked that $$\cap_{j=1}^{\infty} K_j\subset K_{\varepsilon}:=\big\{x\in\Mn: f _{\varepsilon}(x) \geq 1-\varepsilon \big\}.$$  Thus, for all $j$ large enough, one has $K_j\subset K_{\varepsilon}$, and then ${(1-\varepsilon)^{-1}} f_{\varepsilon}\geq 1$ on $K_j$.  
 Together with \eqref{hom-1-1}, Corollary \ref{cpqk-f-smooth} and \eqref{kj plus 1 and kj}, one has
\begin{align*}
\lim _{j \to \infty} C_{p, Q}(K_j)  
&\leq \Big(\int_{S^{nm-1}}\left\|h_Q\Big(\nabla_{\u} \big((1-\varepsilon)^{-1}f _{\varepsilon}\big)\Big)\right\|_p^{-nm} d\u\Big)^{-\frac{p}{nm}}\\
&=(1-\varepsilon)^{-p}\Big(\int_{S^{nm-1}}\left\|h_Q(\nabla_{\u} f _{\varepsilon})\right\|_p^{-nm} d\u\Big)^{-\frac{p}{nm}}\\
&\leq (1-\varepsilon)^{-p}\left(C_{p, Q}\left(\cap_{j=1}^{\infty} K_j\right)+\varepsilon\right).
\end{align*} After taking   $\varepsilon \to 0$, one gets  
\begin{align}\label{upper semi continuity}
\lim _{j \to \infty} C_{p, Q}\left(K_j\right) \leq C_{p, Q}\left(\cap_{j=1}^{\infty} K_j\right).
\end{align} 
On the other hand, a simple argument based on the monotonicity of $C_{p, Q}(\bigcdot)$ and the sequence $\{K_j\}_{j\in \N}$ gives 
$$C_{p, Q}\big(\cap_{j=1}^{\infty} K_j\big) \leq \lim _{j \to \infty} C_{p, Q}\left(K_j\right),$$ where the limit clearly exists and is finite (due to  $C_{p, Q}(K_1)<\infty$).  
Together with  \eqref{upper semi continuity}, one gets the desired continuity from above. 
\end{proof}

\section{On the comparisons of the $m$th order $p$-affine capacity with the $p$-variational capacity and volume}\label{section-4}

This section is dedicated to establishing inequalities that compare the $m$th order $p$-affine capacity with both the $p$-variational capacity and volume. Let $p \in[1, n)$ and $K \subset \Mn$ be a compact set. Define the $p$-variational capacity   $K$  by
\begin{align}\label{CpK}
C_p(K)=\inf_{f \in \mathscr{B}(K)} \int_{\Mn}|\nabla f(x)|^p dx,
\end{align}
where $\mathscr{B}(K)$ is defined above \eqref{CPQ=BK}, namely,  the set consisting of all functions $f \in W_0^{1, p}$ such that $f \geq 1$ on $K$. Note that $\mathscr{B}(K)$ in \eqref{CpK} can be replaced by $\mathscr{B}(K) \cap C_c^{\infty}$ (see, e.g., \cite[P.27–28]{Heinonen 2006}), or alternatively by  $\mathscr{A}(K)$ or $\mathscr{D}(K)$ as well as their intersections with $C_c^{\infty}$. It can be checked that  
\begin{align*}
C_p(aK)=a^{n-p}C_p(K)\ \ \mathrm{for}\ \ a>0.
\end{align*}
It can be derived from \cite[P.148]{Mazya 2011} that  
\begin{align}\label{cpb2n}
C_p(B_2^n)=\begin{cases}
0, &\ \ \mathrm{if}\ \ p\ge n,\\
n \omega_n \big(\frac{n-p}{p-1}\big)^{p-1}, &\ \ \mathrm{if}\ \ p\in (1,n), \\ \lim_{q\rightarrow 1^+}C_q(B_2^n)=n \omega_n, &\ \ \mathrm{if}\ \ p=1.
\end{cases}    
\end{align}
Here, $\omega_{n}$ represents the volume of $B_2^{n}$.
See e.g., \cite[P.171]{Evans-Gariepy-1992} and \cite[P.141]{Mazya 2011} for more information about the $p$-variational capacity. 

The following inequality compares the $m$th order $p$-affine capacity with the volume. For $p\in [1,\infty)$ and $Q \in \mathcal{K}_o^{1, m}$, let
\begin{align}\label{d n p Q}
d_{n, p}(Q):=(n \omega_n)^{-1} \big(n m V_{n m}(\Pi_{p,Q}^{*} B_2^n)\big)^{-\frac{p}{n m}},
\end{align}
where   $\Pi_{p,Q} K$ for $K\in\mathcal{K}_{(o)}^{n,1}$ is defined in \eqref{support function of ho lp projection body} and $\Pi^*_{p,Q} K$ denotes the polar body of $\Pi_{p,Q} K$.  
\begin{theorem}\label{theorem with volume}
Let $p \in(1, n)$, $Q\in\mathcal{K}_o^{1,m}$ and $K \subset \Mn$ be a compact set. Then,
\begin{align}\label{C p Q K ge}
C_{p,Q}(K)\ge n\omega_n^{\frac{p}{n}}d_{n,p}(Q)\Big(\frac{p-1}{n-p}\Big)^{1-p}V_n(K)^{\frac{n-p}{n}}.
\end{align} 
\end{theorem}
\begin{proof}
Let $f\in C_c^{\infty}\cap \mathscr{D}(K)$. It has been proved by Langharst, Roysdon and Zhao in \cite[(3.19)]{mth order affine polya szego principle} that 
\begin{align}\label{mth-polya-szego-3.19}
\Big(\int_{S^{nm-1}}\left\|h_Q\left(\nabla_{\u} f\right)\right\|_p^{-nm} d \u\Big)^{-\frac{p}{nm}}\ge n^p\omega_n^{\frac{p}{n}}d_{n,p}(Q)\int_0^{1}V_n(\left[f\right]_t)^{\frac{p(n-1)}{n}}\Big(-\frac{dV_n(\left[f\right]_t)}{dt}\Big)^{1-p}dt,
\end{align}
where  $\left[f\right]_t$ is defined in \eqref{level-set-f}. As $f\in C_c^{\infty}\cap \mathscr{D}(K)$, one gets $f\in[0,1]$ and $K\subset [f]_1\subset [f]_0$. Hence $V_n(K)\leq V_n(\left[f\right]_1)\leq V_n(\left[f\right]_0)$.  It follows from \eqref{mth-polya-szego-3.19} and Jensen's inequality that
\begin{align*}
\Big(\!\int_{S^{nm-1}}\!\left\|h_Q\left(\nabla_{\u} f\right)\right\|_p^{-nm} d \u\Big)^{-\frac{p}{nm}}
&\ge   n^p\omega_n^{\frac{p}{n}}d_{n,p}(Q)\Big(\int_0^1 V_n(\left[f\right]_t)^{\frac{p(n-1)}{n(1-p)}}\big(-dV_n(\left[f\right]_t)\big)\Big)^{1-p} \\
&= n^p\omega_n^{\frac{p}{n}}d_{n,p}(Q)\bigg(\!\Big(\frac{n(p-1)}{n-p}\Big)\left(V_n(\left[f\right]_1)^{\frac{n-p}{n-np}}\!-\!V_n(\left[f\right]_0)^{\frac{n-p}{n-np}}\right)\!\bigg)^{1-p}\\
&\ge  n\omega_n^{\frac{p}{n}}d_{n,p}(Q)\Big(\frac{p-1}{n-p}\Big)^{1-p}V_n(K)^{\frac{n-p}{n}}.
\end{align*} The desired inequality \eqref{C p Q K ge} follows immediately from Corollary \ref{cpqk-f-smooth} and after taking the infimum over $f\in C_c^{\infty}\cap \mathscr{D}(K)$. \end{proof}

Our second result is the following inequality comparing $C_p(\bigcdot)$ with $C_{p,Q}(\bigcdot)$. 

\begin{theorem}\label{cpqk and cpk}
Let $p\in [1,\infty)$, $Q\in\mathcal{K}_o^{1,m}$ and $K\subset \Mn$ be a compact set. Then,
\begin{align*}
C_{p,Q}(K)\leq d_{n, p}(Q) C_p(K).
\end{align*}
\end{theorem}

\begin{proof}
The proof follows directly from \cite[Theorem 1.2]{mth order affine polya szego principle}, which states:  for $f \in W_0^{1, p}$, $p\in [1,\infty)$ and $Q\in\mathcal{K}_o^{1,m}$, one has
\begin{align*}
\Big(\int_{S^{nm-1}}\left\|h_Q(\nabla_{\u} f)\right\|_p^{-nm} d \u\Big)^{-\frac{p}{nm}}\leq d_{n, p}(Q)\left\|\nabla f\right\|_p^p.
\end{align*}
By Theorem \ref{equiv def of cpqk of ek} and \eqref{CpK}, one immediately obtains the desired inequality after taking the infimum over $f\in \mathscr{B}(K)$. 
\end{proof} 

We shall need the following result.

\begin{proposition}\label{c1qk ge} Let  $Q\in\mathcal{K}_o^{1,m}$ and $K \subset \Mn$ be a compact set. Then, 
    \begin{align*}
C_{1, Q}(K) \ge \limsup_{p\to 1^{+}} C_{p, Q}(K).    
\end{align*} 
\end{proposition}
\begin{proof} It can be easily checked that, for $f\in C_c^{\infty}\cap\mathscr{D}(K)$, $\u \in S^{nm-1}$ and $Q\in\mathcal{K}_o^{1,m}$,  one has 
\begin{align*}
\lim_{p\to 1^{+}}\left\|h_Q\left(\nabla_{\u} f\right)\right\|_p^{-nm}
=&\lim_{p\to 1^{+}} \Big(\int_{\Mn}h_Q\big(\nabla_{\u} f(x)\big)^p d x\Big)^{-\frac{nm}{p}}\\
=&\Big(\int_{\Mn}\lim_{p\to 1^{+}}h_Q\Big(\big(\nabla f(x)\big)^{\mathrm{T}}{}_{\bigcdot}\u\Big)^p dx\Big)^{-{nm}}\!\!=\left\|h_Q\left(\nabla_{\u} f\right)\right\|_1^{-{nm}},
\end{align*} where  the dominated convergence theorem has been used. Employing the Fatou's lemma, one can get 
\begin{align*}
\Big(\int_{S^{nm-1}}\left\|h_Q(\nabla_{\u} f)\right\|_1^{-nm} d \u\Big)^{-\frac{1}{nm}}
&=\Big(\int_{S^{nm-1}}\lim_{p\to 1^{+}}\left\|h_Q(\nabla_{\u} f)\right\|_p^{-nm} d \u\Big)^{-\frac{1}{nm}}\\
&\ge \Big(\liminf_{p\to 1^{+}} \int_{S^{nm-1}}\left\|h_Q(\nabla_{\u} f)\right\|_p^{-nm} d \u\Big)^{-\frac{1}{nm}} \\
& =\limsup_{p\to 1^{+}}\Big(\int_{S^{nm-1}}\left\|h_Q\left(\nabla_{\u} f\right)\right\|_p^{-nm} d \u\Big)^{-\frac{p}{nm}} \\ & \geq \limsup _{p\to 1^{+}} C_{p, Q}(K),     
\end{align*} where the last inequality follows from Corollary \ref{cpqk-f-smooth}. After taking the infimum over $f \in C_c^{\infty} \cap \mathscr{D}(K)$, one gets  the desired result again with the help of Corollary \ref{cpqk-f-smooth}.
\end{proof}
 
We are in the position to prove the following result. A set $K\subset \Mn$ is called an ellipsoid if there exist $\phi\in GL(n)$ and $y_0\in\Mn$  such that $K=\phi{}_{\bigcdot}B^n_2+y_0$.
\begin{theorem}\label{theorem with volume-1}
Let $Q\in\mathcal{K}_o^{1,m}$ and  $K \subset \Mn$ be a compact set. Then, 
\begin{align}\label{C p Q B leq}
C_{p,Q}(B_2^n)= \begin{cases}
0, &\ \ \mathrm{for}\ \ p\ge n,\\
n \omega_n d_{n, p}(Q)\big(\frac{n-p}{p-1}\big)^{p-1}, &\ \ \mathrm{for}\ \ 1<p<n, \\
n \omega_nd_{n, 1}(Q), &\ \ \mathrm{for}\ \ p=1.
\end{cases}   
\end{align}  Moreover, the following inequality holds with equality if $K$ is an ellipsoid:  for $p\in [1, n)$,
\begin{align}\label{cpqk over cpqb}
\bigg(\frac{V_n(K)}{V_n(B_2^n)}\bigg)^{\frac{1}{n}}\leq \bigg(\frac{C_{p, Q}(K)}{C_{p, Q}(B_2^n)}\bigg)^{\frac{1}{n-p}}.
\end{align} 
\end{theorem}

\begin{proof} First of all, an easy consequence of  Theorem \ref{cpqk and cpk} and \eqref{cpb2n} is the following upper bound for $C_{p,Q}(B_2^n)$. That is,  for $Q\in \mathcal{K}_o^{1,m}$,
\begin{align}\label{C p Q B leq-1}
C_{p,Q}(B_2^n)\leq \begin{cases}
0, &\ \ \mathrm{for}\ \ p\ge n,\\
n \omega_n d_{n, p}(Q)\big(\frac{n-p}{p-1}\big)^{p-1}, &\ \ \mathrm{for}\ \ 1<p<n, \\
n \omega_nd_{n, 1}(Q), &\ \ \mathrm{for}\ \ p=1.
\end{cases}   
\end{align} In particular, $C_{p,Q}(B_2^n)=0$ if $p\geq n$. For $p\in (1, n)$, \eqref{C p Q K ge} together with \eqref{C p Q B leq-1} yields that
\begin{align}\label{cpqb}
C_{p,Q}(B_2^n)= n\omega_nd_{n,p}(Q)\Big(\frac{p-1}{n-p}\Big)^{1-p}.    
\end{align}  Thus, \eqref{cpqk over cpqb} holds for $p\in(1,n)$ by \eqref{C p Q K ge} and \eqref{cpqb}.

For $p=1$,  by \eqref{cpqb} and Proposition \ref{c1qk ge},  one gets 
\begin{align}\label{lower-bound-C-1-Q}
C_{1, Q}(B_2^n)\ge \limsup _{p\to 1^{+}} C_{p, Q}(B_2^n)=n \omega_n \lim_{p\rightarrow 1^+} d_{n, p}(Q).    
\end{align}
Now we claim that \begin{align}\label{continuity-d-n-p-Q} 
\lim_{p\rightarrow 1^+}d_{n, p}(Q)= d_{n,1}(Q).
\end{align} 
To this end, we first show that, for any nonempty compact set $G \subset \Mm$,
\begin{align}\label{h-Q-j-to-h-Q}
\max_{z\in G}\Big|h_Q^{p}(z)-h_Q(z)\Big|\rightarrow 0\ \ \mathrm{as}\ \ p\rightarrow 1^+.
\end{align}  Since $G$ is compact and $h_Q$ is continuous, it follows that $h_Q$ is bounded on $G$. Hence, there exists $R>1$ such that $0\leq h_Q(z)\leq R$ for all $z\in G$. Let 
\begin{align*}
H_p(t)=|t^p-t|,\ \ \ t\in [0,R].
\end{align*}
Then, to prove  \eqref{h-Q-j-to-h-Q}, it suffices to show that $H_p(\bigcdot)\rightarrow 0$ uniformly on $[0,R]$ as $p\rightarrow 1^+$. First of all, if $t\in [0, 1]$, then $H_p(t)=t-t^p$ and $H_p(0)=H_p(1)=0$. Furthermore, for $t\in [0, 1]$, $H_p(t)$ attains its maximum at  $t_0=p^{\frac{1}{1-p}}$. That is, 
\begin{align*}
 \max_{t\in [0,1]}H_p(t)=  H_p(t_0)= p^{\frac{1}{1-p}}(1-p^{-1}).
\end{align*} For each $p>1$, one gets $H_p(t)=t^p-t$ on $t\in [1, R]$, which is strictly increasing on $[1, R]$. Thus, $\max_{t\in [1, R]} H_p(t)=R^p-R$, and then   
\begin{align*}
 \max_{t\in [0, R]} H_p(t) = \max\Big\{ R^p-R,\  p^{\frac{1}{1-p}}(1-p^{-1})\Big\}. 
\end{align*} The uniform convergence of $H_p(\bigcdot)$ on $[0, R]$ follows directly by the easily established limits: 
$$\lim_{p\rightarrow 1^+} R^p-R=0 \ \ \mathrm{and} \ \   \lim_{p\rightarrow 1^+} p^{\frac{1}{1-p}}(1-p^{-1})=0.$$ This completes the proof of \eqref{h-Q-j-to-h-Q}. 

Recall from \cite[P.10]{the-mth-order-Orlicz-projection-bodies} that, for any $p\ge 1$ and $Q\in\mathcal{K}_o^{1,m}$, the function $h_Q^p(\bigcdot):\Mm\rightarrow[0,\infty)$ is continuous and convex, such that,  $h_Q^p(o)=0$ and 
$$h_Q^p(z)+h_Q^p(-z)>0\  \ \ \mathrm{for}\  \ z\ne o.$$
Combining with \eqref{support function of ho lp projection body}, \eqref{h-Q-j-to-h-Q} and \cite[Proposition 4.2]{the-mth-order-Orlicz-projection-bodies}, one gets  
$$h_{\Pi_{p,Q}K}(\bigcdot)\rightarrow h_{\Pi_{1,Q}K}(\bigcdot) \ \ \mathrm{uniformly\  \ on\ } \ S^{nm-1} \ \ \mathrm{as}\ \ p\rightarrow 1^+.$$
By \eqref{VnmKstar} and the fact that $\Pi_{p,Q}K\in\mathcal{K}_{(o)}^{n,m}$ see \cite[Proposition 3.2]{Haddad-ye-lp-2025}), one gets
\begin{align*}
V_{nm}\big(\Pi^*_{p,Q}K\big)\rightarrow V_{nm}\big(\Pi^*_{1,Q}K\big)\ \ \mathrm{as}\ \ p\rightarrow 1^+.
\end{align*}
By letting $K=B_2^n$ and applying \eqref{d n p Q}, one gets \eqref{continuity-d-n-p-Q}.
By \eqref{C p Q B leq-1}, \eqref{lower-bound-C-1-Q}  and \eqref{continuity-d-n-p-Q}, one gets
\begin{align}\label{C1QB2n}
C_{1, Q}(B_2^n)=\limsup _{p\to 1^{+}} C_{p, Q}(B_2^n)=n \omega_n d_{n, 1}(Q).   
\end{align} Note that \eqref{cpqk over cpqb} holds for $p\in(1,n)$. It can be easily checked by Proposition \ref{c1qk ge} and \eqref{C1QB2n} that 
\begin{align*}
\left(\frac{V_n(K)}{V_n(B_2^n)}\right)^{\frac{1}{n}}\leq \limsup _{p\to 1^{+}}\left(\frac{C_{p, Q}(K)}{C_{p, Q}(B_2^n)}\right)^{\frac{1}{n-p}}\leq \left(\frac{C_{1, Q}(K)}{C_{1, Q}(B_2^n)}\right)^{\frac{1}{n-1}}. 
\end{align*} This shows that \eqref{cpqk over cpqb} also holds for $p=1$.

Clearly, for $p\in[1,n)$, equality in \eqref{cpqk over cpqb} holds for $K=B_2^n$, and hence for $K$ being an ellipsoid, due to the translation invariance and the affine invariance of the $m$th order $p$-affine capacity (see Proposition \ref{proposition 1}). This completes the proof.  
\end{proof}

Recall that  $\mathcal{L}_o^{n,m}$ is the set of Lipschitz star bodies about $o$ in $\M$.  Let $K \in \mathcal{L}_o^{n,1}$ and   
\begin{align*}
\partial^* K=\Big\{x \in \partial K:\ \nabla \rho_K(x) \ \mathrm{exists}  \Big\}.
\end{align*}
It follows from Lin and Xi \cite[Lemma 2.1]{lin and xi} that   
\begin{align}\label{null-set} \mathcal{H}^{n-1}(\partial K)<\infty, \ \ 
\mathcal{H}^{n-1}\left(\partial K \setminus\partial^*K\right)=0,
\end{align}
and there exist positive constants $a_K$ and $b_K$ such that, for all $z \in \partial^* K$,
\begin{align}\label{z-cdot-nukz-range}
a_K \leq z \bigcdot \nu_K(z) \leq b_K. 
\end{align}
For $p\in[1,\infty)$, define the $L_p$ surface area of $K\in\mathcal{L}_o^{n,1}$ by
\begin{align}\label{lp-surface-area-of-star body}
S_p(K)=\int_{\partial^*K}\big(z\bigcdot \nu_{K}(z)\big)^{1-p}d\mathcal{H}^{n-1}(z).
\end{align}
Combining \eqref{cpb2n}, \eqref{cpqk over cpqb}, Theorem \ref{cpqk and cpk} and \cite[Lemma 1]{Ludwig Xiao and Zhang 2011}, we obtain the following inequality:  for $p\in [1,n)$ and $K\subset \Mn$ being a Lipschitz star body,
\begin{align*}
\left(\frac{V_n(K)}{V_n(B_2^n)}\right)^{\frac{1}{n}}
\leq
\left(\frac{C_{p, Q}(K)}{C_{p, Q}(B_2^n)}\right)^{\frac{1}{n-p}}
\leq
\left(\frac{C_p(K)}{C_p(B_2^n)}\right)^{\frac{1}{n-p}}
\leq
\left(\frac{S_p(K)}{S_p(B_2^n)}\right)^{\frac{1}{n-p}}. 
\end{align*}
Furthermore, the first two inequalities hold for $K$ being a compact set.

\section{On the comparison of the $m$th order $p$-affine capacity with  the $m$th order $p$-integral affine surface area}\label{Section-5}

In this section, we aim to establish inequalities comparing the $m$th order $p$-affine capacity with  the $m$th order $p$-integral affine surface area. To this end, we first introduce the $m$th order $p$-integral affine surface area for Lipschitz star bodies, along with some of its basic properties. 

We begin by extending the $(L_p, Q)$-projection body defined for convex bodies (see \eqref{support function of ho lp projection body} or \cite[Definition 1.2]{Haddad-ye-lp-2025}) to Lipschitz star bodies.  
 
\begin{definition}\label{h-pi-Q-K}
Let $p\in [1,\infty)$, $Q\in\mathcal{K}_{o}^{1,m}$ and $K \in \mathcal{L}_o^{n,1}$. The $(L_p,Q)$-projection body of $K$ is denoted by $\Pi_{p,Q}K$ and defined via the function $h_{\Pi_{p,Q}K}:\M\rightarrow [0,\infty)$ formulated by 
\begin{align*}
h_{\Pi_{p,Q}K}(\x)=\bigg(\int_{\partial^* K} h_Q\left(\nu_K^{\mathrm{T}}(z) {}_{\bigcdot}\x\right)^p\big(z \bigcdot \nu_K(z)\big)^{1-p} d\mathcal{H}^{n-1}(z)\bigg)^{\frac{1}{p}}\ \ \mathrm{for}\ \  \x\in\M. 
\end{align*}
\end{definition} An alternative definition of $\Pi_{p,Q}K$ for $K\in\mathcal{L}_o^{n,1}$, derived from \cite[(4.5)]{lin and xi}, is to express $h_{\Pi_{p,Q}K}$ as an integral over $K$:
\begin{align}\label{equiv-h-pi-Q-K}
h_{\Pi_{p,Q}K}(\x)=\Big(n\int_{K} h_Q\left(\nabla p_K^{\mathrm{T}}(y) {}_{\bigcdot}\x\right)^p dy\Big)^{\frac{1}{p}} \ \ \mathrm{for}\ \ \x\in\M, 
\end{align}
where $p_K$ is given in \eqref{PL}.  Note that by \cite[Lemma 2.1]{lin and xi}, $\nabla p_K(y)=-\frac{\nabla \rho_K(y)}{(\rho_K(y))^2}$ exists for almost all $y\in\Mn\setminus \{o\}$.

Before presenting the properties of $\Pi_{p,Q}K$, we first recall two properties of the Orlicz projection body for a Lipschitz star body in $\Mn$ about $o$ (see \cite[Lemma 6.1]{lin and xi} and \cite[Lemma 6.2]{lin and xi}).  Let $\varphi: \R\rightarrow [0,\infty)$ be a convex function such that $\varphi(0)=0$ and $\varphi$ is either strictly decreasing on $(-\infty, 0]$ or $\varphi$ is strictly increasing on $[0,\infty)$.
For $z \in \Mn \setminus\{0\}$, one has $h_{\Pi_{\varphi} K}(z)=\lambda_0>0$ if and only if
\begin{align}\label{def-Orlicz-projection-LSB}
\frac{1}{nV_n(K)} \int_{\partial^* K} \varphi\left(\frac{z \bigcdot \nu_K(x)}{\lambda_0 x \bigcdot \nu_K(x)}\right) x \bigcdot \nu_K(x) d \mathcal{H}^{n-1}(x)=1.  
\end{align}
Moreover, one has
\begin{align}\label{Lower-bound-projection-body-LSB}
 h_{\Pi_\varphi K}(u)\ge \frac{1}{2 n c_\varphi R_K} \quad \forall u \in S^{n-1}, 
\end{align}
where $R_K=\max_{u\in S^{n-1}}\rho_K(u)$ and 
\begin{align}\label{c-varphi}
c_{\varphi}=\max\{c>0: \max\{\varphi(c), \varphi(-c)\}\leq 1\}.
\end{align}

We now establish some properties for $\Pi_{p,Q}K$  needed in later context. For $\phi\in GL(n)$, denote by $\phi^{-\mathrm{T}}$ the transpose of the inverse of $\phi$, that is,
$\phi^{-\mathrm{T}}=(\phi^{-1})^{\mathrm{T}}.$ 
\begin{proposition}\label{proposition-h-pi-Q-p-K}
Let $p \geq 1$, $Q\in\mathcal{K}_{o}^{1,m}$ and $K \in \mathcal{L}_o^{n,1}$. The following statements hold.
\vskip 1mm \noindent i)  Finiteness: 
for $\u\in S^{nm-1}$, $h_{\Pi_{p,Q}K}(\u)$ is finite and positive.
\vskip 1mm \noindent ii) Sublinearity:
$h_{\Pi_{p,Q}K}$ is sublinear, and hence,
$\Pi_{p,Q}K\in\mathcal{K}_{(o)}^{n,m}$.
\vskip 1mm \noindent iii) Homogeneity:
 $\Pi_{p,bQ}(aK)=ba^{\frac{n-p}{p}}\Pi_{p,Q}K$ for   $a, b>0$.
\vskip 1mm \noindent iv) 
Affine invariance: for $\phi\in GL(n)$, one has
\begin{align}\label{affine property of projection body for convex body}
\Pi_{p,Q}(\phi{}_{\bigcdot}K)=|\det (\phi)|^{\frac{1}{p}}\phi^{-\mathrm{T}}{}_{\bigcdot}\Pi_{p,Q}K.    
\end{align}
\end{proposition}
\begin{proof}
\vskip 1mm \noindent i) Let $R_Q$ be the constant such that $Q\subset R_Q B_2^m$. Note that $|v^{\mathrm{T}}{}_{\bigcdot}\u|\leq 1$ for $v\in S^{n-1}$ and $\u\in S^{nm-1}$. Hence, there exists $y\in Q$ such that
\begin{align*}
h_Q(v^{\mathrm{T}}{}_{\bigcdot}\u)=y\bigcdot (v^{\mathrm{T}}{}_{\bigcdot}\u)\leq |y|\leq R_Q.
\end{align*}
Together with \eqref{volume-K}, \eqref{null-set}, \eqref{z-cdot-nukz-range}, Definition \ref{h-pi-Q-K} and  $Q\in\mathcal{K}_o^{1,m}$, one gets,  for any $\u\in S^{nm-1}$ and $p\geq 1$,   \begin{align*}
h_{\Pi_{p,Q}K}(\u)^p
&=\int_{\partial^* K} h_Q\left(\nu_K^{\mathrm{T}}(z) {}_{\bigcdot}\u\right)^p\big(z \bigcdot \nu_K(z)\big)^{1-p} d\mathcal{H}^{n-1}(z)\\
&\leq \int_{\partial^* K}   a_K^{-p}R_Q^p \ z \bigcdot \nu_K(z)d\mathcal{H}^{n-1}(z)\\
&=na_K^{-p}R_Q^p V_n(K)<\infty.\end{align*} 

We now show that  $h_{\Pi_{p,Q}K}(\u)>0$ for all $\u\in S^{nm-1}$. As $Q\in\mathcal{K}_o^{1,m}$, for each $\u\in S^{nm-1}$, there exists a point $q_{\u}\in Q$ such that $\u{}_{\bigcdot}q_{\u}^{\mathrm{T}}\ne o$. It follows from \eqref{h KB}, \eqref{null-set}, \eqref{z-cdot-nukz-range}, Definition \ref{h-pi-Q-K} and Jensen's inequality that
\begin{align*}
h_{\Pi_{p,Q}K}(\u)
&=\Big(\int_{\partial^* K} h_Q\big(\nu_K^{\mathrm{T}}(z) {}_{\bigcdot}\u\big)^p\big(z \bigcdot \nu_K(z)\big)^{1-p} d\mathcal{H}^{n-1}(z)\Big)^{\frac{1}{p}}\\
&\ge \Big(\int_{\partial^* K} h_Q\left(\nu_K^{\mathrm{T}}(z) {}_{\bigcdot}\u\right)^pb_K^{1-p} d\mathcal{H}^{n-1}(z)\Big)^{\frac{1}{p}}\\
&= b_K^{\frac{1-p}{p}}V_{n-1}(\partial^* K)^{\frac{1}{p}}\bigg(\int_{\partial^* K} h_{\u{}_{\bigcdot}Q^{\mathrm{T}}}\big(\nu_K(z) \big)^p \frac{d\mathcal{H}^{n-1}(z)}{V_{n-1}(\partial^* K)}\bigg)^{\frac{1}{p}}\\
&\ge b_K^{\frac{1-p}{p}}V_{n-1}(\partial^* K)^{\frac{1-p}{p}}\int_{\partial^* K} h_{\u{}_{\bigcdot}Q^{\mathrm{T}}}\big(\nu_K(z) \big) d\mathcal{H}^{n-1}(z)\\
&\ge b_K^{\frac{1-p}{p}}V_{n-1}(\partial^* K)^{\frac{1-p}{p}}\int_{\partial^* K} h_{[o,\u{}_{\bigcdot}q_{\u}^{\mathrm{T}}]}\big(\nu_K(z) \big) d\mathcal{H}^{n-1}(z)\\&=b_K^{\frac{1-p}{p}}V_{n-1}(\partial^* K)^{\frac{1-p}{p}}\int_{\partial^* K}  \Big(\big(\u{}_{\bigcdot}q_{\u}^{\mathrm{T}}\big) \bigcdot  \nu_K(z) \Big)_+ d\mathcal{H}^{n-1}(z)\\ 
&=b_K^{\frac{1-p}{p}}V_{n-1}(\partial^* K)^{\frac{1-p}{p}}h_{\Pi_{\varphi} K}(\u{}_{\bigcdot}q_{\u}^{\mathrm{T}}),
\end{align*}
where $h_{\Pi_{\varphi} K}(\bigcdot)$ is defined in \eqref{def-Orlicz-projection-LSB} with $\varphi(t)=t_+$ for $t\in \R$. Together with \eqref{Lower-bound-projection-body-LSB} and \eqref{c-varphi}, one has, $c_{\varphi}=1$ and
$$h_{\Pi K}(\u{}_{\bigcdot}q_{\u}^{\mathrm{T}})\ge \frac{1}{2nR_K}.$$
Therefore, $h_{\Pi_{p,Q}K}(\u)$ is strictly positive for every $\u\in S^{nm-1}$.

\vskip 2mm \noindent ii) 
The sub-additivity follows directly from the sub-linearity of $h_Q$ for $Q\in\mathcal{K}_{o}^{1,m}$ and the Minkowski inequality for $L_p$ norms: for $\x, \y\in\M$,  \begin{align*}
h_{\Pi_{p,Q}K}(\x+\y)&=\bigg(\int_{\partial^* K} h_Q\left(\nu_K^{\mathrm{T}}(z) {}_{\bigcdot} (\x+\y)\right)^p\big(z \bigcdot \nu_K(z)\big)^{1-p} d\mathcal{H}^{n-1}(z)\bigg)^{\frac{1}{p}}\\ &\leq \bigg(\int_{\partial^* K} \big[h_Q\left(\nu_K^{\mathrm{T}}(z) {}_{\bigcdot} \x\right)+h_Q\left(\nu_K^{\mathrm{T}}(z) {}_{\bigcdot} \y\right)\big]^p\big(z \bigcdot \nu_K(z)\big)^{1-p} d\mathcal{H}^{n-1}(z)\bigg)^{\frac{1}{p}}\\&\leq \bigg(\int_{\partial^* K} \big[h_Q\left(\nu_K^{\mathrm{T}}(z) {}_{\bigcdot} \x\right)\big]^p\big(z \bigcdot \nu_K(z)\big)^{1-p} d\mathcal{H}^{n-1}(z)\bigg)^{\frac{1}{p}}  \\  &\ \ \ +\bigg(\int_{\partial^* K} \big[h_Q\left(\nu_K^{\mathrm{T}}(z) {}_{\bigcdot} \y\right)\big]^p\big(z \bigcdot \nu_K(z)\big)^{1-p} d\mathcal{H}^{n-1}(z)\bigg)^{\frac{1}{p}}\\&= 
h_{\Pi_{p,Q}K}(\x)+h_{\Pi_{p,Q}K}(\y).
\end{align*} 
The homogeneity of degree $1$ of $h_{\Pi_{p,Q}K}(\bigcdot)$ can be obtained directly from that of $h_Q(\bigcdot)$. Thus, $h_{\Pi_{p,Q}K}(\bigcdot)$ is sublinear.
Note that a sublinear function uniquely determines a convex set. Together with i), one gets
$\Pi_{p,Q}K\in\mathcal{K}_{(o)}^{n,m}$.

\vskip 2mm \noindent iii) Let $a>0$, $b>0$ and let $z=ay$ for $y\in\partial K$. Thus, $z\in\partial (aK)$ and $\nu_{aK}(z)=\nu_K(y)$. It follows from \eqref{hom-1-1} and Definition \ref{h-pi-Q-K} that, for $\x\in \M$,
\begin{align*}
h_{\Pi_{p,bQ}aK}(\x)^p&=\int_{\partial^* (aK)} h_{bQ}\big(\nu_{aK}^{\mathrm{T}}(z) {}_{\bigcdot}\x\big)^p\big(z \bigcdot \nu_{aK}(z)\big)^{1-p} d\mathcal{H}^{n-1}(z) \\
&=\int_{\partial^* K} b^ph_{Q}\big(\nu_{K}^{\mathrm{T}}(y) {}_{\bigcdot}\x\big)^p\big(ay \bigcdot \nu_{K}(y)\big)^{1-p} a^{n-1}d\mathcal{H}^{n-1}(y) \\
&=b^pa^{n-p}h_{\Pi_{p,Q}K}(\x)^p=h_{\big(ba^{\frac{n-p}{p}}\Pi_{p,Q}K\big)}(\x)^p.
\end{align*}
This yields $\Pi_{p,bQ}(aK)=ba^{\frac{n-p}{p}}\Pi_{p,Q}K$ for   $a, b>0$.
\vskip 1mm \noindent iv) The result for $K\in\mathcal{K}_{o}^{n,1}$ has been established in \cite[Proposition 3.6]{Haddad-ye-lp-2025}. For $K\in\mathcal{L}_o^{n,1}$, we first consider the case $\phi\in GL(n)$ with $|\det(\phi)|=1$. Let $z\in \partial K$ and $\widetilde{z}=\phi{}_{\bigcdot}z$. By \cite[(6.4)]{lin and xi}, one gets,  if $\nabla p_K(z)$ exists, then
\begin{align}\label{nabda-p-Tk}
\nabla p_{\phi{}_{\bigcdot}K}(\widetilde{z})=\phi^{-\mathrm{T}}{}_{\bigcdot}\nabla p_K(z).
\end{align}
Recall that $\nabla p_K(z)$ exists for almost all $z\in\partial K$.
Combining with \eqref{h KB}, \eqref{equiv-h-pi-Q-K} and \eqref{nabda-p-Tk}, one gets, for $\x\in \M$,
\begin{align*}
h_{\Pi_{p,Q}\phi{}_{\bigcdot}K}(\x)&=\Big(n\int_{\phi{}_{\bigcdot}K} h_Q\Big(\nabla p_{\phi{}_{\bigcdot}K}^{\mathrm{T}}(\widetilde{z}) {}_{\bigcdot}\x\Big)^p d\widetilde{z}\Big)^{\frac{1}{p}}\\
&=\Big(n\int_{K} h_Q\Big(\big(\phi^{-\mathrm{T}}{}_{\bigcdot}\nabla p_K(z)\big)^{\mathrm{T}}{}_{\bigcdot}\x\Big)^p dz\Big)^{\frac{1}{p}}\\
&=\Big(n\int_K h_Q\Big(\nabla p_K(z)^{\mathrm{T}}{}_{\bigcdot}(\phi^{-1}{}_{\bigcdot}\x)\Big)^p dz\Big)^{\frac{1}{p}}\\
&=h_{\Pi_{p,Q}K}(\phi^{-1}{}_{\bigcdot}\x)=h_{\phi^{-\mathrm{T}}{}_{\bigcdot}\Pi_{p,Q}K}(\x).
\end{align*}
Thus, for $\phi\in GL(n)$ with $|\det(\phi)|=1$, one has $\Pi_{p,Q}\phi{}_{\bigcdot}K=\phi^{-\mathrm{T}}{}_{\bigcdot}\Pi_{p,Q}K.$

Now suppose $\phi\in GL(n)$. Then $\phi$ can be written as 
\begin{align*}
\phi=|\det(\phi)|^{\frac{1}{n}}\widehat{\phi},
\end{align*}
where $\widehat{\phi}\in GL(n)$ satisfies $|\det(\widehat{\phi})|=1$.
Together with the homogeneity that $\Pi_{p,Q}aK=a^{\frac{n-p}{p}}\Pi_{p,Q}K$ for $a>0$, one gets 
\begin{align*}
\Pi_{p,Q}\phi{}_{\bigcdot}K
&=\Pi_{p,Q}|\det(\phi)|^{\frac{1}{n}}\widehat{\phi}{}_{\bigcdot}K
=|\det(\phi)|^{\frac{n-p}{np}}\Pi_{p,Q}\widehat{\phi}{}_{\bigcdot}K\\
&=|\det(\phi)|^{\frac{n-p}{np}}\widehat{\phi}^{-\mathrm{T}}{}_{\bigcdot}\Pi_{p,Q}K=|\det(\phi)|^{\frac{1}{p}}\phi^{-\mathrm{T}}{}_{\bigcdot}\Pi_{p,Q}K.
\end{align*}
This completes the proof of \eqref{affine property of projection body for convex body}.
\end{proof}

With the help of $\Pi_{p,Q}K$, where $K\in\mathcal{L}_o^{n,1}$ is a Lipschitz star body about $o$, one can define the $m$th order $p$-integral affine surface area $\Phi_{p,Q}(K)$ as follows. 

\begin{definition}\label{def of Phi p Q}
Let $p\in [1,\infty)$, $Q\in\mathcal{K}_{o}^{1,m}$ and $K \in \mathcal{L}_o^{n,1}$. The $m$th order $p$-integral affine surface area of $K$ is defined by
\begin{align}\label{phi-p-Q-K-wrt-volume}
\Phi_{p,Q}(K)=\big(n m V_{n m}(\Pi_{p,Q}^{*} K)\big)^{-\frac{p}{nm}}=\Big(\int_{S^{nm-1}}h^{-nm}_{\Pi_{p,Q}K}(\u) d\u\Big)^{-\frac{p}{nm}}.
\end{align}
\end{definition} It follows from  \eqref{d n p Q} and \eqref{phi-p-Q-K-wrt-volume}  that 
\begin{align}\label{phi pq b2n}
\Phi_{p,Q}(B_2^n)=n\omega_n d_{n, p}(Q).
\end{align}
Let $\varphi_{\tau}$ be given in \eqref{varphi tau} with $|\tau|\leq 1$. If $m=1$ and 
$$Q=\Big[-\Big(\frac{1-\tau}{2}\Big)^{\frac{1}{p}},\ \Big(\frac{1+\tau}{2}\Big)^{\frac{1}{p}}\Big],$$ then  $h_Q^p(t)=\varphi_{\tau}(t)$ for $t\in\Mm$,
and $\Phi_{p,Q}(K)$ reduces to the general $p$-integral affine surface area in \cite[(5.7)]{Hong and Ye 2018}. Furthermore, for $\tau=0$ and $K\in\mathcal{K}_{o}^{n,1}$, it reduces to the $p$-integral affine surface area, see, e.g., \cite{lp proj ineq 1,Petty-projection-bodies,Petty isoperimetric problem}. 

To better understand the $m$th order $p$-integral affine surface area, we now provide some of its fundamental properties. It is well known that, for $K\in\mathcal{K}_{(o)}^{n,1}$ and  $\phi\in GL(n)$,
\begin{align}\label{ck-star}
(\phi{}_{\bigcdot}\, K)^*=\phi^{-\mathrm{T}}{}_{\bigcdot}\, K^*,
\end{align} 
where $K^*$ is the polar body of $K$ defined in \eqref{polar-body}.
\begin{proposition}\label{proposition-for-Phi}
Let $p \geq 1$, $Q\in\mathcal{K}_{o}^{1,m}$ and $K \in \mathcal{L}_o^{n,1}$. The following statements hold.
\vskip 1mm \noindent i)  
Homogeneity: $\Phi_{p, bQ}(aK)=b^pa^{n-p} \Phi_{p, Q}(K)$ for $a>0$ and $b>0$.
\vskip 1mm \noindent ii)  
Symmetry: $\Phi_{p, Q}(K)=\Phi_{p,-Q}(K)$.
\vskip 1mm \noindent iii)  
Concavity: for any $\lambda \in[0,1]$ and $Q_1, Q_2\in\mathcal{K}_{o}^{1,m}$,
$$
\Phi_{p, \lambda\bigcdot_p Q_1+_p(1-\lambda)\bigcdot_p  Q_2}(K) \geq \lambda  \Phi_{p, Q_1}(K)+(1-\lambda)  \Phi_{p, Q_2}(K),
$$ where $\lambda\bigcdot_p Q_1+_p(1-\lambda)\bigcdot_p  Q_2$ denotes the $L_p$ sum of $Q_1$ and $ Q_2$  defined in \eqref{lp-addition-with-lp-multi}. 

\vskip 1mm \noindent iv)  
Affine invariance:   
\begin{align*}
\Phi_{p, Q}(\phi{}_{\bigcdot}K)=|\det (\phi)|^{\frac{n-p}{n}} \Phi_{p, Q}(K) \ \ \mathrm{for\ } \ \ \phi \in G L(n).
\end{align*}
\end{proposition}

\begin{proof}
\vskip 1mm \noindent i) The homogeneity of $\Phi_{p,Q}(K)$ follows directly from Proposition \ref{proposition-h-pi-Q-p-K} iii) and Definition \ref{def of Phi p Q}.

\vskip 1mm \noindent ii) 
By \eqref{h KB} and Definition \ref{h-pi-Q-K}, for any $\u\in S^{nm-1}$ and $p\geq 1$, one has \begin{align*}
h_{\Pi_{p,-Q}K}(\u)^p
&=\int_{\partial^* K} h_{-Q}\left(\nu_K^{\mathrm{T}}(z) {}_{\bigcdot}\u\right)^p\big(z \bigcdot \nu_K(z)\big)^{1-p} d\mathcal{H}^{n-1}(z)\\
&=\int_{\partial^* K} h_{-Q{}_{\bigcdot}\u^{\mathrm{T}}}\left(\nu_K^{\mathrm{T}}(z) \right)^p\big(z \bigcdot \nu_K(z)\big)^{1-p} d\mathcal{H}^{n-1}(z)\\
&=\int_{\partial^* K} h_{Q{}_{\bigcdot}(-\u)^{\mathrm{T}}}\left(\nu_K^{\mathrm{T}}(z) \right)^p\big(z \bigcdot \nu_K(z)\big)^{1-p} d\mathcal{H}^{n-1}(z)\\
&=\int_{\partial^* K} h_Q\left(\nu_K^{\mathrm{T}}(z) {}_{\bigcdot}(-\u)\right)^p\big(z \bigcdot \nu_K(z)\big)^{1-p} d\mathcal{H}^{n-1}(z)=h_{\Pi_{p,Q}K}(-\u)^p.
\end{align*}
Combining with Definition \ref{def of Phi p Q}, one has
\begin{align*}
\Phi_{p,-Q}(K)&=\Big(\int_{S^{nm-1}}h_{\Pi_{p,-Q}K}(\u)^{-nm} d\u\Big)^{-\frac{p}{nm}}=\Big(\int_{S^{nm-1}}h_{\Pi_{p,Q}K}(-\u)^{-nm} d\u\Big)^{-\frac{p}{nm}}\\
&=\Big(\int_{S^{nm-1}}h_{\Pi_{p,Q}K}(\v)^{-nm} d\v\Big)^{-\frac{p}{nm}}=\Phi_{p,Q}(K).
\end{align*}
This completes the proof of ii).
\vskip 1mm \noindent iii)
By \eqref{lp-addition-with-lp-multi} and Definition \ref{h-pi-Q-K}, one obtains that, for all $\u\in S^{nm-1}$, $p\geq 1$, $\lambda \in[0,1]$ and $Q_1, Q_2\in\mathcal{K}_{o}^{1,m}$,  \begin{align*}
&h_{\Pi_{p,\lambda\bigcdot_p Q_1+_p(1-\lambda)\bigcdot_p  Q_2}K}(\u)^p\\
&=\int_{\partial^* K} h_{\lambda\bigcdot_p Q_1+_p(1-\lambda)\bigcdot_p  Q_2}\left(\nu_K^{\mathrm{T}}(z) {}_{\bigcdot}\u\right)^p\big(z \bigcdot \nu_K(z)\big)^{1-p} d\mathcal{H}^{n-1}(z)\\
&=\int_{\partial^* K}\Big(\lambda h_{Q_1}\left(\nu_K^{\mathrm{T}}(z) {}_{\bigcdot}\u\right)^p +(1-\lambda)h_{Q_2}\left(\nu_K^{\mathrm{T}}(z) {}_{\bigcdot}\u\right)^p\Big)\big(z \bigcdot \nu_K(z)\big)^{1-p} d\mathcal{H}^{n-1}(z)\\
&=\lambda \int_{\partial^* K}h_{Q_1}\left(\nu_K^{\mathrm{T}}(z) {}_{\bigcdot}\u\right)^p\big(z \bigcdot \nu_K(z)\big)^{1-p} d\mathcal{H}^{n-1}(z)\\
&~~+(1-\lambda)\int_{\partial^* K}h_{Q_2}\left(\nu_K^{\mathrm{T}}(z) {}_{\bigcdot}\u\right)^p\big(z \bigcdot \nu_K(z)\big)^{1-p} d\mathcal{H}^{n-1}(z)\\
&=\lambda h_{\Pi_{p,Q_1}K}(\u)^p+(1-\lambda) h_{\Pi_{p,Q_2}K}(\u)^p.
\end{align*}
Together with the reverse Minkowski inequality and Definition \ref{def of Phi p Q}, one gets
\begin{align*}
\Phi_{p,\lambda\bigcdot_p Q_1+_p(1-\lambda)\bigcdot_p  Q_2}(K)&=\Big(\int_{S^{nm-1}}h_{\Pi_{p,\lambda\bigcdot_p Q_1+_p(1-\lambda)\bigcdot_p  Q_2}K}(\u)^{-nm} d\u\Big)^{-\frac{p}{nm}}\\
&=\Big(\int_{S^{nm-1}} \Big(\lambda h_{\Pi_{p,Q_1}K}(\u)^p+(1-\lambda) h_{\Pi_{p,Q_2}K}(\u)^p\Big)^{-\frac{nm}{p}} d\u\Big)^{-\frac{p}{nm}}\\
&\ge \lambda \Big(\!\int_{S^{nm-1}}\!\!\!\!h_{\Pi_{p,Q_1}K}(\u)^{-nm}d\u\!\Big)^{-\frac{p}{nm}}\!\!\!+(1-\lambda)\Big(\!\int_{S^{nm-1}}\!\!\!\!\!h_{\Pi_{p,Q_2}K}(\u)^{-nm}d\u\!\Big)^{-\frac{p}{nm}}\\
&=\lambda\Phi_{p,Q_1}(K)+(1-\lambda)\Phi_{p,Q_2}(K).
\end{align*}
This completes the proof of the desired concavity.

\vskip 1mm \noindent iv)
It follows from \eqref{ck-star} and Proposition \ref{proposition-h-pi-Q-p-K} iv) that, for any $\phi \in G L(n)$, $p\ge 1$ and $K \in \mathcal{L}_o^{n,1}$, 
$$
\Pi_{p,Q}^{*} (\phi{}_{\bigcdot}K)=|\det(\phi)|^{-\frac{1}{p}}\phi{}_{\bigcdot}\Pi_{p,Q}^{*}K,
$$
which further gives
$$V_{n m}\big(\Pi_{p,Q}^{*} (\phi{}_{\bigcdot}K)\big)=|\det(\phi)|^{\frac{pm-nm}{p}}V_{n m}(\Pi_{p,Q}^{*}K).$$
Together with \eqref{phi-p-Q-K-wrt-volume}, one has 
\begin{align*}
\Phi_{p, Q}(\phi{}_{\bigcdot}K)=\Big(n m V_{n m}\big(\Pi_{p,Q}^{*} (\phi{}_{\bigcdot}K)\big)\Big)^{-\frac{p}{nm}}=|\det (\phi)|^{\frac{n-p}{n}} \Phi_{p, Q}(K),   
\end{align*}
as desired. This completes the proof.

\end{proof}

Notice that there shall have no translation invariance for $\Phi_{p,Q}(\bigcdot)$. The inequality to compare  the $m$th order $p$-affine capacity $C_{p,Q}(K)$ and the $m$th order $p$-integral affine surface area $\Phi_{p,Q}(K)$ for $p\in(1,n)$ and $K\in\mathcal{L}_o^{n,1}$ is established below.  
\begin{theorem}\label{theorem c and phi when p ge 1}
Let $p\in(1,n)$, $Q\in\mathcal{K}_o^{1,m}$ and  $K\in\mathcal{L}_o^{n,1}$. Then,
\begin{align}\label{p ge 1 C and Phi}
\frac{C_{p, Q}(K)}{C_{p, Q}(B_2^n)} \leq \frac{\Phi_{p, Q}(K)}{\Phi_{p, Q}(B_2^n)},
\end{align}
with equality if $K$ is an origin-symmetric ellipsoid. 
\end{theorem}
\begin{proof} The proof shall follow the techniques in \cite[Theorem 5.3]{Hong and Ye 2018} and \cite[P.962]{Xiao 2016}. 
For $p\in(1, n)$, define $g:[0, \infty)\to \mathbb{R}$ by
\begin{align*}
g(s)=\begin{cases}
1, &\ \ \mathrm{for}\ \ s\in[0,1],\\
s^{\frac{n-p}{1-p}}, &\ \ \mathrm{for}\ \ s\in(1,\infty).
\end{cases}
\end{align*} Clearly, $0<g(s)< 1$ for $s>1$, and  $g$ is strictly decreasing on $(1, \infty)$. As $\rho_K (x)<1$ if $x\notin K$ and $\rho_K(x)\ge1$  for $x\in K,$ one has,  
\begin{align*}
f(x)=g\Big(\frac{1}{\rho_K(x)}\Big)=\begin{cases}
1, &\ \ \mathrm{for}\ \ x\in K,\\
\rho_K(x)^\frac{n-p}{p-1}, &\ \ \mathrm{for}\ \ x\notin K.
\end{cases}   
\end{align*} Clearly, $\mathbf{1}_K(x) \leq f(x)\leq 1$ for all $x\in\Mn$ and hence  $f\in\mathscr{A}(K)$. 

Let $s\in (1,\infty)$ and $t=g(s)$ for $s>1$, then $t\in(0,1)$.  As $g$ is strictly decreasing on $(1,\infty)$, by \eqref{level-set-f}, one gets,
\begin{align}\label{ft equals sK}
[f]_{t}=\bigg\{x \in \Mn: g\Big(\frac{1}{\rho_K(x)}\Big) \ge g(s)\bigg\}=\bigg\{x \in  \Mn: \frac{1}{\rho_K(x)}\leq s\bigg\}=s K.
\end{align} It can be checked (see e.g., \cite[Lemma 6]{Ludwig Xiao and Zhang 2011}) that, for almost all $x \in \partial[f]_t=\partial (sK)$, there exists $z \in \partial K$ with $x=sz$ (and hence $s=\frac{1}{\rho_K(x)}$), such that
\begin{align}\label{nu-K-nu-F}
|\nabla f(x)|=\frac{\left|g^{\prime}(s)\right|}{|z \bigcdot \nu_K(z)|} \ \ \mathrm{and}\ \  \nu_K(z)=\nu_{[f]_{t}}(x)=-\frac{\nabla f(x)}{|\nabla f(x)|}.
\end{align} 
Combining with \eqref{hom-1-1}, \eqref{nu-f-t}, \eqref{Federer-coarea-formula}, \eqref{support function of ho lp projection body}, \eqref{null-set}, \eqref{z-cdot-nukz-range}, Definition \ref{h-pi-Q-K}, \eqref{ft equals sK} and \eqref{nu-K-nu-F}, one has 
\begin{align*}
\|h_Q(\nabla_{\u} f)\|_p^p &=\int_{\Mn}h_Q\big(\nabla f^{\mathrm{T}}(x) {}_{\bigcdot} \u\big)^p dx\\
&=\int_0^1\int_{\partial[f]_t}h_Q\big(\nabla f^{\mathrm{T}}(x) {}_{\bigcdot} \u\big)^p |\nabla f(x)|^{-1} d\mathcal{H}^{n-1}(x) dt \\
&=\int_0^1\int_{\partial[f]_t} h_Q\big(\!-\nu_{[f]_t}^{\mathrm{T}}(x){}_{\bigcdot}\u\big)^p|\nabla f(x)|^{p-1} d\mathcal{H}^{n-1}(x) dt \\
&=\int_1^{\infty}\int_{\partial^* K} h_Q\big(\nu_K^{\mathrm{T}}(z) {}_{\bigcdot}(-\u)\big)^p|g^{\prime}(s)|^p \big(z \bigcdot \nu_K(z)\big)^{1-p} s^{n-1} d\mathcal{H}^{n-1}(z) ds\\
&=\bigg(\int_1^{\infty}|g^{\prime}(s)|^p s^{n-1} ds\bigg) \bigg(\int_{\partial^* K} h_Q\big(\nu_K^{\mathrm{T}}(z) {}_{\bigcdot}(-\u)\big)^p\big(z \bigcdot \nu_K(z)\big)^{1-p} d\mathcal{H}^{n-1}(z)\bigg)\\
&= h_{\Pi_{p,Q}K}(-\u)^p\Big(\frac{n-p}{p-1}\Big)^p\int_1^{\infty}s^{\frac{n-1}{1-p}}ds=\Big(\frac{n-p}{p-1}\Big)^{p-1}h_{\Pi_{p,Q}K}(-\u)^p.
\end{align*}
Together with Definition \ref{def of cpqk with ek} and Definition \ref{def of Phi p Q}, one gets
\begin{align*}
C_{p, Q}(K)&\leq\Big(\int_{S^{nm-1}} \| h_Q\left(\nabla_{\u} f\right)\|_p^{-n m} d \u\Big)^{-\frac{p}{n m}}\\
&\leq \Big(\frac{n-p}{p-1}\Big)^{p-1}\Big(\int_{S^{nm-1}}h^{-nm}_{\Pi_{p,Q}K}(\u) d\u\Big)^{-\frac{p}{nm}}
=\Big(\frac{n-p}{p-1}\Big)^{p-1} \Phi_{p,Q}(K).
\end{align*}
By \eqref{cpqb} and \eqref{phi pq b2n}, one gets \eqref{p ge 1 C and Phi}. Clearly, the equality in the above  inequality  holds if $K=B_2^n$. Due to the affine invariance of $C_{p, Q}(K)$ (see Proposition \ref{proposition 1}) and $\Phi_{p, Q}(K)$ (see Proposition \ref{proposition-for-Phi}), equality then  holds in \eqref{p ge 1 C and Phi} if $K$ is an origin-symmetric ellipsoid.
\end{proof}

We now provide the inequality between  $C_{1,Q}(\bigcdot )$ and  $\Phi_{1,Q}(\bigcdot)$.  
\begin{theorem}\label{theorem c and phi when p equal 1}
Let $Q\in\mathcal{K}_o^{1,m}$ and  $K\in\mathcal{L}_o^{n,1}$. Then
\begin{align}\label{c1Qk and Phi 1 QK}
\frac{C_{1, Q}(K)}{C_{1, Q}(B_2^n)} \leq \frac{\Phi_{1, Q}(K)}{\Phi_{1, Q}(B_2^n)}, 
\end{align}
with equality if $K$ is an origin-symmetric ellipsoid.  
\end{theorem}
\begin{proof} By \eqref{C1QB2n} and \eqref{phi pq b2n}, one has $C_{1, Q}(B_2^n)=\Phi_{1, Q}(B_2^n)$. Thus,  \eqref{c1Qk and Phi 1 QK} holds if $C_{1,Q}(K)\leq \Phi_{1,Q}(K).$ To this end, for small $\varepsilon>0$,  let 
\begin{align*}
K_{\varepsilon}=\Big\{x\in\Mn: 0<\dist(x, K)<\varepsilon\Big\}. 
\end{align*}  We also define $f_{\varepsilon}\in\mathscr{D}(K)$ by \begin{align*}
f_{\varepsilon}(x)=\begin{cases}
1-\frac{\dist(x, K)}{\varepsilon}, &\ \ \mathrm{if}\ \ \dist(x, K)<\varepsilon,\\
0, &\ \ \mathrm{otherwise}.
\end{cases}   
\end{align*} By \cite[P. 195]{Zhang 1999}, for each $x\in K_{\varepsilon}$, there exists a unique $x^{\prime} \in \partial K$ such that
$\mathrm{dist}(x, K)=|x'-x|,$ and hence if $x\in K_{\varepsilon}$,
\begin{align}\label{nabda-f-varepsilon}
\nabla f_{\varepsilon}(x)=\frac{\nu_K(x^{\prime})}{\varepsilon},
\end{align}
where $\nu_K(x')=\frac{x'-x}{|x'-x|}$. Moreover, $\nabla f_{\varepsilon}(x)=0$ if $x\notin \overline{K}_{\varepsilon}$. For each $x\in K_{\varepsilon}$, define $t=\dist(x, K)$ and then $0<t<\varepsilon$. It follows from   \eqref{hom-1-1}, \eqref{null-set} and \eqref{nabda-f-varepsilon} that
\begin{align*}
\lim_{\varepsilon\to 0}\int_{\Mn}h_Q(\nabla f_{\varepsilon}(x)^{\mathrm{T}}{}_{\bigcdot}\u)dx
=&\lim_{\varepsilon\to 0}\int_{K_{\varepsilon}}h_Q(\varepsilon^{-1} \nu_K(x^{\prime})^{\mathrm{T}}{}_{\bigcdot}\u)dx\\
=&\lim_{\varepsilon\to 0}\bigg(\varepsilon^{-1}\int_0^{\varepsilon}\int_{\partial K}h_Q\big(\nu_K(x')^{\mathrm{T}}{}_{\bigcdot}\u\big)d\mathcal{H}^{n-1}(x')dt+ \varepsilon^{-1}o(\varepsilon)\bigg)\\
=&\int_{\partial^* K}h_Q\big(\nu_K(x')^{\mathrm{T}}{}_{\bigcdot}\u\big)d\mathcal{H}^{n-1}(x^{\prime}),
\end{align*}where we have used $dx=d\mathcal{H}^{n-1}(x')d t+o(t)dt$, with $o(t)$ a higher-order term that tends to $0$ as $t\rightarrow0$, and satisfying $$\frac{1}{\varepsilon} \int_0^{\varepsilon} o(t) d t=\frac{1}{\varepsilon} o(\varepsilon)\rightarrow 0 \ \mathrm{as}\ \ \varepsilon\rightarrow 0.$$  By Definition \ref{def of cpqk with ek}, \eqref{null-set}, Definition \ref{h-pi-Q-K}, Definition \ref{def of Phi p Q} and Fatou's lemma, one gets,
\begin{align*} 
C_{1,Q}(K) & \leq \limsup _{\varepsilon \to 0}\bigg(\int_{S^{nm-1}}\Big(\int_{\Mn} h_Q\big(\nabla f_{\varepsilon}(x)^{\mathrm{T}}{}_{\bigcdot}\u\big) d x\Big)^{-nm} d\u\bigg)^{-\frac{1}{nm}} \\
& =\bigg(\liminf _{\varepsilon \to 0} \int_{S^{nm-1}}\Big(\int_{\Mn} h_Q\big(\nabla f_{\varepsilon}(x)^{\mathrm{T}}{}_{\bigcdot}\u\big) d x\Big)^{-nm} d\u\bigg)^{-\frac{1}{nm}}\\
&\leq\bigg(\int_{S^{nm-1}}\lim_{\varepsilon \to 0} \Big(\int_{\Mn} h_Q\big(\nabla f_{\varepsilon}(x)^{\mathrm{T}}{}_{\bigcdot}\u\big) d x\Big)^{-nm} d\u\bigg)^{-\frac{1}{nm}}\\
&=\bigg(\int_{S^{nm-1}} \Big(\int_{\partial^* K}h_Q\big(\nu_K(x')^{\mathrm{T}}{}_{\bigcdot}\u\big)d\mathcal{H}^{n-1}(x')\Big)^{-nm} d\u\bigg)^{-\frac{1}{nm}}\\
&=\bigg(\int_{S^{nm-1}} h_{\Pi_{1,Q}K}(\u)^{-nm} d\u\bigg)^{-\frac{1}{nm}}=\Phi_{1,Q}(K).
\end{align*} 
Thus, the desired inequality \eqref{c1Qk and Phi 1 QK} holds. Moreover, equality holds in \eqref{c1Qk and Phi 1 QK} if $K=B_2^n$. By the affine invariance of $C_{1,Q}(K)$ and $\Phi_{1,Q}(K)$, equality holds in \eqref{c1Qk and Phi 1 QK} if $K$ is an origin-symmetric ellipsoid.  \end{proof}

The inequality between $\Phi_{p,Q}(K)$ and $S_p(K)$ will be established in the following theorem. Recall from \eqref{bar T u} that $\phi{}_{\bigcdot}\u=(\phi{}_{\bigcdot}u_1,\cdots,\phi{}_{\bigcdot}u_m)$ for $\phi\in GL(n)$ and  $\u=(u_1,\cdots, u_m)\in \M$ with $u_i\in\Mn$ for $i=1,\cdots,m$. Denote by $O(n)$ the orthogonal group, consisting of all $n\times n$ orthogonal matrices. That is, if $\phi\in O(n)$, then $\phi$ is an $n\times n$ matrix such that $\phi^{\mathrm{T}}_{\bigcdot} \phi=\phi_{\bigcdot} \phi^{\mathrm{T}}$ is the identity matrix on $M_{n, 1}(\R)$. The natural measure on $O(n)$ is the normalized Haar measure denoted by $\mu$. Note that $\mu$ is a rotational invariant probability measure. Thus,  for any $\u\in S^{nm-1}$,
\begin{align}\label{O-N}
\int_{O(n)} h_Q\Big(\big(\phi^{\mathrm{T}}{}_{\bigcdot} v\big)^{\mathrm{T}} {}_{\bigcdot}\u\Big)^p d\mu(\phi)
\end{align}
is independent of $v\in S^{n-1}$. 
\begin{proposition}\label{theorem phi and s}
For $p\in [1,n)$ and $K$ in $\mathcal{L}_o^{n,1}$, one has
\begin{align*}
\frac{\Phi_{p,Q}(K)}{\Phi_{p,Q}(B_2^n)}\leq\frac{S_p(K)}{S_p(B_2^n)},
\end{align*} with equality if $K$ is an origin-symmetric ball. 
\end{proposition}
\begin{proof} It follows from Definition \ref{h-pi-Q-K} and Definition \ref{def of Phi p Q} that, for $\phi\in O(n)$,  one has 
\begin{align*}
\Phi_{p,Q}(K)^{-\frac{nm}{p}}
&=\int_{S^{nm-1}}\Big(\int_{\partial^* K} h_Q\big(\nu_K^{\mathrm{T}}(z) {}_{\bigcdot}\u\big)^p\big(z \bigcdot \nu_K(z)\big)^{1-p} d \mathcal{H}^{n-1}(z) \Big)^{-\frac{nm}{p}}d\u\\
&=\int_{S^{nm-1}}\Big(\int_{\partial^* K} h_Q\big(\nu_K^{\mathrm{T}}(z) {}_{\bigcdot}(\phi{}_{\bigcdot}\u)\big)^p\big(z \bigcdot \nu_K(z)\big)^{1-p} d \mathcal{H}^{n-1}(z) \Big)^{-\frac{nm}{p}}d\u\\
&=\int_{S^{nm-1}}\bigg(\int_{\partial^* K} h_Q\Big(\big(\phi^{\mathrm{T}}{}_{\bigcdot}\nu_K(z)\big)^{\mathrm{T}}{}_{\bigcdot}\u\Big)^p\big(z \bigcdot \nu_K(z)\big)^{1-p} d \mathcal{H}^{n-1}(z)\bigg)^{-\frac{nm}{p}}d\u,
\end{align*} where the second equality follows from the rotational invariance of the spherical measure on $S^{nm-1}$. 

By \eqref{lp-surface-area-of-star body}, \eqref{null-set}, \eqref{z-cdot-nukz-range}, Definition \ref{def of Phi p Q}, Fubini theorem, Jensen's inequality and the fact that the integral \eqref{O-N} is independent of $v$,  one gets
\begin{align*}
\Phi_{p,Q}(K)^{-\frac{nm}{p}}\!\!
&=\int_{O(n)}\int_{S^{nm-1}}\bigg(\int_{\partial^* K} h_Q\Big(\big(\phi^{\mathrm{T}}{}_{\bigcdot}\nu_K(z)\big)^{\mathrm{T}}{}_{\bigcdot}\u\Big)^p\big(z \bigcdot \nu_K(z)\big)^{1-p} d \mathcal{H}^{n-1}(z) \bigg)^{-\frac{nm}{p}}d\u d\mu(\phi)\\
&=\int_{S^{nm-1}}\int_{O(n)} \bigg(\int_{\partial^* K} h_Q\Big(\big(\phi^{\mathrm{T}}{}_{\bigcdot}\nu_K(z)\big)^{\mathrm{T}}{}_{\bigcdot}\u\Big)^p\big(z \bigcdot \nu_K(z)\big)^{1-p} d \mathcal{H}^{n-1}(z)\bigg)^{-\frac{nm}{p}}d\mu(\phi)d\u\\
&\ge\int_{S^{nm-1}}\bigg(\int_{O(n)} \int_{\partial^* K} h_Q\Big(\big(\phi^{\mathrm{T}}{}_{\bigcdot}\nu_K(z)\big)^{\mathrm{T}}{}_{\bigcdot} \u\Big)^p\big(z \bigcdot \nu_K(z)\big)^{1-p} d \mathcal{H}^{n-1}(z)d\mu(\phi)\bigg)^{-\frac{nm}{p}}d\u \\
&=\int_{S^{nm-1}}\bigg( \int_{\partial^* K} \int_{O(n)}h_Q\Big(\big(\phi^{\mathrm{T}}{}_{\bigcdot}\nu_K(z)\big)^{\mathrm{T}}{}_{\bigcdot} \u\Big)^pd\mu(\phi) \big(z \bigcdot \nu_K(z)\big)^{1-p} d \mathcal{H}^{n-1}(z) \bigg)^{-\frac{nm}{p}}d\u\\
&=\bigg(\!\!\int_{\partial^* K} \big(z \bigcdot \nu_K(z)\big)^{1-p} d \mathcal{H}^{n-1}(z)\!\!\bigg)^{-\frac{nm}{p}}\!\!\!\!\int_{S^{nm-1}}\!\!\!\bigg( \int_{O(n)}h_Q\Big(\big(\phi^{\mathrm{T}}{}_{\bigcdot}\nu_K(z)\big)^{\mathrm{T}}{}_{\bigcdot} \u\Big)^pd\mu(\phi) \!\! \bigg)^{-\frac{nm}{p}}\!\!\!d\u\\
&=S_p(K)^{-\frac{nm}{p}}\int_{S^{nm-1}}\bigg(\frac{1}{n\omega_n} \int_{S^{n-1}}h_Q\big(v^{\mathrm{T}}{}_{\bigcdot} \u\big)^pdv\bigg)^{-\frac{nm}{p}}d\u,
\end{align*}
where the last equation follows from \eqref{lp-surface-area-of-star body} and \cite[(3.3)]{mth order affine polya szego principle}.
Together with \eqref{VnmKstar},  \eqref{support function of ho lp projection body}, \eqref{d n p Q} and \eqref{lp-surface-area-of-star body}, one has, 
\begin{align*}
\Phi_{p,Q}(K)^{-\frac{nm}{p}}
&\geq S_p(K)^{-\frac{nm}{p}}(n\omega_n)^{\frac{nm}{p}}\int_{S^{nm-1}}h_{\Pi_{p,Q}B_2^n}(\u)^{-nm} d\u\\
&=S_p(K)^{-\frac{nm}{p}}(n\omega_n)^{\frac{nm}{p}} nmV_{nm}(\Pi^*_{p,Q}B_2^n)\\
&=S_p(K)^{-\frac{nm}{p}} d_{n,p}(Q)^{-\frac{nm}{p}}.
\end{align*} Taking the power of $-\frac{p}{mn}$ from both sides, one gets $$\Phi_{p,Q}(K)\leq S_p(K)d_{n,p}(Q).$$  The desired inequality follows from a rearrangement and  \eqref{phi pq b2n}.
Clearly, equality holds if $K$ is an origin-symmetric ball. This completes the proof.
\end{proof}

Theorem \ref{theorem with volume-1}, along with Theorem \ref{theorem c and phi when p ge 1}, Theorem \ref{theorem c and phi when p equal 1} and Lemma \ref{theorem phi and s} yields that, for $p\in [1,n)$, $K\in\mathcal{L}_o^{n,1}$ and $Q\in\mathcal{K}_o^{1,m}$, one gets:
\begin{align*}
\left(\frac{V_n(K)}{V_n(B_2^n)}\right)^{\frac{1}{n}}\leq \left(\frac{C_{p, Q}(K)}{C_{p, Q}(B_2^n)}\right)^{\frac{1}{n-p}}
\leq \left(\frac{\Phi_{p, Q}(K)}{\Phi_{p, Q}(B_2^n)}\right)^{\frac{1}{n-p}}\leq \left(\frac{S_p(K)}{S_p(B_2^n)}\right)^{\frac{1}{n-p}}.
\end{align*}

\vskip 2mm \noindent  {\bf Acknowledgement.}  The
research of DY was supported by a NSERC grant, Canada.  

\vspace{16pt}

\noindent Xia Zhou, Department of Mathematics and Statistics, Memorial University of Newfoundland, St. John’s, Newfoundland, A1C 5S7, Canada\\
\textit{Email address}: xiaz@mun.ca
\vspace{8pt}

\noindent Deping Ye, Department of Mathematics and Statistics, Memorial University of Newfoundland, St. John’s, Newfoundland, A1C 5S7, Canada\\
\textit{Email address}: deping.ye@mun.ca
\vspace{8pt}


\begin{thebibliography}{99}


\bibitem{sharp-lp-sobolev-Aubin-p-ge-1}
T. Aubin, \textit{Problèmes isopérimétriques et espaces de Sobolev}, J. Differential Geom., \textbf{11} (1976), 573-598.


\bibitem{Borell 1983}
C. Borell, \textit{Capacitary inequalities of the Brunn-Minkowski type}, Math. Ann., \textbf{263} (1983), 179-184.


\bibitem{Caffarelli 1996}
L. A. Caffarelli, D. Jerison and E. H. Lieb, \textit{On the case of equality in the Brunn-Minkowski inequality for capacity}, Adv. Math., \textbf{117} (1996), 193-207.


\bibitem{Cianchi-LYZ-2009}
A. Cianchi, E. Lutwak, D. Yang and G. Zhang, \textit{Affine Moser-Trudinger and Morrey-Sobolev inequalities}, Calc. Var. Partial Differential Equations, \textbf{36} (2009), 419-436. 

\bibitem{Colesanti 2015}
A. Colesanti, K. Nyström, P. Salani, J. Xiao, D. Yang and G. Zhang, \textit{The Hadamard variational formula and the Minkowski problem for $p$-capacity}, Adv. Math., \textbf{285} (2015), 1511-1588.

\bibitem{for p capacity 1}
A. Colesanti and P. Salani, \textit{The Brunn-Minkowski inequality for $p$-capacity of convex bodies}, Math. Ann., \textbf{327} (2003), 459-479.


\bibitem{Evans-Gariepy-1992}
L. Evans and R. Gariepy, \textit{Measure Theory and Fine Properties of Functions}, Studies in Advanced Mathematics, CRC Press, Boca Raton, (1992).

\bibitem{Federer 1969}
H. Federer, \textit{Geometric Measure Theory}, Springer, Berlin, (1969).

\bibitem{Federer-p-1}
H. Federer and W. Fleming, \textit{Normal and integral currents,} Ann. Math., \textbf{72} (1960), 458-520.

\bibitem{Gardner-2002-LP-isop-Ine}
R. J. Gardner, \textit{The Brunn-Minkowski inequality}, Bull. Amer. Math. Soc., \textbf{39} (2002), 355-405.

\bibitem{Guillemin-1974}
V. Guillemin and A. Pollack, \textit{Differential Topology}, Prentice-Hall, Englewood Cliffs, NJ, (1974).

\bibitem{lp-affine-isoperimetric-ine}
C. Haberl and F. E. Schuster, \textit{General $L_p$ affine isoperimetric inequalities}, J. Differ. Geom., \textbf{83} (2009), 1-26.

\bibitem{asym affine lp sobo ine}
C. Haberl and F. E. Schuster, \textit{Asymmetric affine $L_p$ Sobolev inequalities}, J. Funct. Anal., \textbf{257} (2009), 641-658.

\bibitem{Haberl-Schuster-Xiao-2012}
C. Haberl, F. E. Schuster and J. Xiao, \textit{An asymmetric affine Pólya-Szegö principle}, Math. Ann., \textbf{352} (2012), 517-542.

\bibitem{J. Haddad}
J. Haddad, \textit{A Rogers-Brascamp-Lieb-Luttinger inequality in the space of matrices}, arXiv:2309.13298 (2023).

\bibitem{lp affine appl 2}
 J. Haddad, C. Jiménez and M. Montenegro, \textit{Sharp affine Sobolev type inequalities via the $L_p$ Busemann-Petty centroid inequality}, J. Funct. Anal., \textbf{271} (2016), 454-473.

\bibitem{Haddad-Putterman-2025}
J. Haddad, D. Langharst, G. V. Livshyts and E. Putterman, \textit{On the polar of Schneider's difference body}, arXiv:2503.06191v3 (2025).

\bibitem{Haddad-Ye-2023}
J. Haddad, D. Langharst, E. Putterman, M. Roysdon and D. Ye, \textit{Affine isoperimetric inequalities for higher-order projection and centroid Bodies}, arXiv:2304.07859 (2023).

\bibitem{Haddad-ye-lp-2025}
J. Haddad, D. Langharst, E. Putterman, M. Roysdon and D. Ye,
\textit{Higher order $L_p$ isoperimetric and sobolev inequalities}, J. Funct. Anal., \textbf{288} (2025), Paper No. 110722, 45 pp.

 \bibitem{Haddad-Ludwig-affine-frac-lp-sobolev-ine}
J. Haddad and M. Ludwig, \textit{
Affine fractional $L^p$ Sobolev inequalities,}
Math. Ann., \textbf{388} (2024), 1091-1115. 

 \bibitem{Haddad-Ludwig-affine-frac-sobolev-isop-ine}
J. Haddad and M. Ludwig, \textit{Affine fractional Sobolev and isoperimetric Inequalities,}
J. Differential Geom., \textbf{129} (2025), 695-724. 

\bibitem{Heinonen 2006}
J. Heinonen, T. Kilpeläinen and O. Martio, \textit{Nonlinear Potential Theory of Degenerate Elliptic Equations}, Dover Publications, Mineola, (2006).

\bibitem{Hong and Ye 2018}
H. Hong and D. Ye, \textit{Sharp geometric inequalities for the general $p$-affine capacity}, J. Geom. Anal., \textbf{28} (2018), 2254-2287.
\bibitem{Hong-Ye-Zhang-2018}
H. Hong, D. Ye and N. Zhang, \textit{The $p$-capacitary Orlicz-Hadamard variational formula and Orlicz-Minkowski problems}, Calc. Var. Partial Differential Equations, \textbf{57} (2018), Paper No. 5, 31 pp.
\bibitem{Hu 2024}
J. Hu, \textit{The $L_p$-Brunn-Minkowski inequalities for variational functionals with $0 \leq p<1$},
arXiv:2409.20269 (2024).
\bibitem{Jerison 1996}
D. Jerison, \textit{A Minkowski problem for electrostatic capacity}, Acta Math., \textbf{176} (1996), 1-47.

\bibitem{Jerison-1996-2}
D. Jerison, \textit{The direct method in the calculus of variations for convex bodies}, Adv. Math., \textbf{122} (1996), 262-279. 

\bibitem{Ji-dual-Minkowski-p-for-p-capacity}
L. Ji, \textit{The dual Minkowski problem for $p$-capacity}, J. Geom. Anal., \textbf{34} (2024), Paper No. 161, 30 pp.

\bibitem{Dylan-2025-Some-Comments}
D. Langharst, \textit{Some comments on the $m$th-order projection bodies}, arXiv:2504.18933v1 (2025).

\bibitem{Dylan-2025-moment-entropy}
D. Langharst, \textit{On Moment-Entropy inequalities in the space of matrices}, arXiv:2503.00451 (2025).

\bibitem{mth order weighted projection body}
D. Langharst, E. Putterman, M. Roysdon and D. Ye, \textit{On the $mth$-order weighted projection body operator and related inequalities}, Pure Appl. Funct. Anal., to appear.

\bibitem{mth order affine polya szego principle}
D. Langharst, M. Roysdon and Y. Zhao, \textit{On the $m$th-order affine Pólya-Szegö principle}, arXiv:2409.02232v2 (2024).

\bibitem{higher order reverse isoperimetric inequality}
D. Langharst, F. Sola and J. Ulivelli, \textit{Higher-order reserve isoperimetric inequalities for log-concave functions,} arXiv:2403.05712 (2024).

\bibitem{higher order lp mean zonoids}
D. Langharst and D. Xi, \textit{General higher order $L_p$ mean zonoids}, Proc. Amer. Math. Soc., \textbf{152} (2024), 5299-5311.

\bibitem{Liu-Sheng-flow-Minkowski-problem-q-capacity}
X. Liu and W. Sheng, \textit{
A curvature flow to the $L_p$ Minkowski-type problem of $q$-capacity,}
Adv. Nonlinear Stud., \textbf{23} (2023), Paper No. 20220040, 21 pp.

\bibitem{lin and xi}
Y. Lin and D. Xi, \textit{Orlicz affine isoperimetric inequalities for star bodies}, Adv. in Appl. Math., \textbf{134} (2022), Paper No. 102308, 32 pp.

\bibitem{Ludwig Xiao and Zhang 2011}
M. Ludwig, J. Xiao and G. Zhang, \textit{Sharp convex Lorentz-Sobolev inequalities}, Math. Ann., \textbf{350} (2011), 169-197.

\bibitem{Lutwak1}
E. Lutwak, \textit{The Brunn-Minkowski-Firey theory. I. Mixed volumes and the Minkowski problem}, J. Differential Geom., \textbf{38} (1993), 131-150.

\bibitem{lp proj ineq 1}
E. Lutwak, D. Yang and G. Zhang, \textit{$L_p$ affine isoperimetric inequalities}, J. Differential Geom., \textbf{56} (2000), 111-132.

\bibitem{LYZ-2000-a-new ellipsoid}
E. Lutwak, D. Yang and G. Zhang, \textit{A new ellipsoid associated with convex bodies}, Duke Math. J., \textbf{104} (2000), 375-390. 

\bibitem{sharpaffine lp sobol ine}
E. Lutwak, D. Yang and G. Zhang, \textit{Sharp affine $L_p$ Sobolev inequalities}, J. Differential Geom., \textbf{62} (2002), 17-38.

\bibitem{Maz'ya-p-large-1}
V. Maz'ya, \textit{Classes of domains and imbedding theorems for function spaces,} Dokl. Akad. Nauk SSSR, \textbf{133} (1960), 527-530.

\bibitem{Mazya 1985}
V. Maz'ya, \textit{Sobolev Spaces}, Springer, Berlin, (1985).
\bibitem{Mazya 2011}
V. Maz'ya, \textit{Sobolev Spaces with Applications to Elliptic Partial Differential Equations}, 2nd edn. Springer, New York, (2011).

\bibitem{Nguyen-2016}
V. H. Nguyen, \textit{New approach to the affine Pólya-Szegö principle and the stability version of the affine Sobolev inequality}, Adv. Math., \textbf{302} (2016), 1080-1110.

\bibitem{Petty-projection-bodies}
C. M. Petty, \textit{Projection bodies}, Proceedings of the Colloquium on Convexity, Copenhagen, 1965, Kobenhavns Univ. Mat. Inst., Copenhagen, (1967), 234-241.

\bibitem{Petty isoperimetric problem}
C. M. Petty, \textit{Isoperimetric problems}, Proceedings of the Conference on Convexity and Combinatorial Geometry, University of Oklahoma, Norman, Department of Mathematics, University of Oklahoma, (1971), 26-41.

\bibitem{Schneider 0}
R. Schneider, \textit{Eine verallgemeinerung des differenzenkörpers}, Monatsh. Math., \textbf{74} (1970), 258-272.


\bibitem{schneider}
R. Schneider, \textit{Convex Bodies: The Brunn-Minkowski Theory}, Second edition, Encyclopedia of Mathematics and its Applications, Cambridge University Press, Cambridge, (2014).

\bibitem{Sola-on-general-versions-of-projection-ine}
F. M. Sola, \textit{On general versions of the Petty projection inequality}, arXiv:2503.00949 (2025).

\bibitem{Talenti-p-large-1}
G. Talenti, \textit{Best constant in Sobolev inequality,} Ann. Mat. Pura Appl., \textbf{110} (1976), 353-372.

\bibitem{Wang-tuo-2012}
T. Wang, \textit{The affine Sobolev-Zhang inequality on $B V(\mathbb{R}^n)$}, Adv. Math., \textbf{230} (2012), 2457-2473.
\bibitem{Wang-2013}
T. Wang, \textit{The affine Pólya-Szegö principle: equality cases and stability}, J. Funct. Anal., \textbf{265} (2013), 1728-1748.

\bibitem{Wang-2015}
T. Wang, \textit{On the discrete functional $L_p$ Minkowski problem}, Int. Math. Res. Not., \textbf{2015}  (2015), 10563-10585.

\bibitem{Xiao 2015}
J. Xiao, \textit{Corrigendum to ``The sharp Sobolev and isoperimetric inequalities split twice"}, Adv. Math., \textbf{268} (2015), 906-914.
\bibitem{Xiao 2016}
J. Xiao, \textit{The p-affine capacity}, J. Geom. Anal., \textbf{26} (2016), 947-966.

\bibitem{Xiao-2017}
J. Xiao, \textit{The p-affine capacity redux}, J. Geom. Anal., \textbf{27} (2017), 2872-2888. 
\bibitem{Xiao-Zhang-2016}
J. Xiao and N. Zhang, \textit{The relative p-affine capacity,} Proc. Amer. Math. Soc., \textbf{144} (2016), 3537-3554.

\bibitem{Xiong-Xiong-Lu-2019}
G. Xiong, J. Xiong and L. Xu, \textit{The $L_p$ capacitary Minkowski problem for polytopes}, J. Funct. Anal., \textbf{277} (2019), 3131-3155.

\bibitem{Zhang 1999}
G. Zhang, \textit{The affine Sobolev inequality}, J. Differential Geom., \textbf{53} (1999), 183-202.
\bibitem{the-mth-order-Orlicz-projection-bodies}
X. Zhou, D. Ye and Z. Zhang, \textit{The $m$th order Orlicz projection bodies}, arXiv:2501.07565, (2025).
\bibitem{Zou and Xiong 2018}
D. Zou and G. Xiong, \textit{A unified treatment for $L_p$ Brunn-Minkowski type inequalities}, Comm. Anal. Geom., \textbf{26} (2018), 435-460.
\bibitem{zou-xiong-2020}
D. Zou and G. Xiong, \textit{The $L_p$ Minkowski problem for the electrostatic p-capacity},
J. Differential Geom., \textbf{116} (2020), 555-596.
\end{thebibliography}
\end{document}